\documentclass[11pt]{amsart}
\usepackage{amsmath}
\usepackage{amsthm}
\usepackage{amssymb}
\usepackage{amsfonts}
\usepackage{dsfont}
\usepackage{url}
\usepackage{graphicx}
\usepackage{xypic}
\usepackage{color}
\usepackage[
        colorlinks,
        backref,
]{hyperref}
\newcounter{tthm}
\newtheorem{thmm}{Theorem}[tthm]
\newtheorem{thm}{Theorem}[section]
\newtheorem{lem}[thm]{Lemma}
\newtheorem{prop}[thm]{Proposition}
\newtheorem{cor}[thm]{Corollary}
\newtheorem{defn}[thm]{Definition}

\newtheorem{remark}[thm]{Remark}

\newcommand{\ideal}[1]{\mathfrak{{#1}}}

\newcommand{\bZ}{{\mathbb Z}} 
\newcommand{\bQ}{{\mathbb Q}}
\newcommand{\bR}{{\mathbb R}}

\newcommand{\bN}{{\mathbb N}}

\newcommand{\dsP}{\mathds P}
\newcommand{\dsF}{\mathds F}
\newcommand{\dsN}{\mathds N}
\newcommand{\dsQ}{\mathds Q}
\newcommand{\dsC}{\mathds C}
\newcommand{\dsD}{\mathds D}
\newcommand{\dsZ}{\mathds Z}

\newcommand{\ord}{\operatorname{ord}}

\newcommand{\chara}{\operatorname{char}}

\newcommand{\Frob}{\operatorname{Frob}}
\newcommand{\Gln}{\operatorname{GL_n}}
\newcommand{\Gal}{\operatorname{Gal}}
\newcommand{\W}{\widetilde{W}}
\newcommand{\G}{\mathcal G}

\newcommand{\ck}{\overline{k}}

\makeindex

\begin{document}
\renewcommand{\theenumi}{\roman{enumi}}
\title[Random Galois groups]{The Galois group of random elements \\of linear groups}
\author{Alexander Lubotzky, Lior Rosenzweig}
\email{alexlub@math.huji.ac.il,rosenzwe@math.huji.ac.il}
\thanks{The authors are supported by ERC, ISF and NSF} %
\begin{abstract}
Let $\dsF$ be a finitely generated field of characteristic zero and $\Gamma\leq\Gln(\dsF)$ a finitely generated subgroup. For $\gamma\in\Gamma$, let $\Gal(\dsF(\gamma)/\dsF)$ be the Galois group of the splitting field of the characteristic polynomial of $\gamma$ over $\dsF$. We show that the structure of $\Gal(\dsF(\gamma)/\dsF)$ has a typical behaviour depending on $\dsF$, and on the geometry of the Zariski closure of $\Gamma$ (but not on $\Gamma$).
\end{abstract}
\maketitle
\section{Introduction}
 Let $F$ be a field of characteristic zero, and $\Gamma$ a finitely generated subgroup of $GL_n(F)$. For an element $\gamma\in \Gln(F)$, we denote by $\chi_\gamma(T)$ the characteristic polynomial of $\gamma$, $F(\gamma)$ is the splitting field of $\chi_\gamma$ over $F$, and $\Gal(F(\gamma)/F)$ the Galois group of $F(\gamma)$ above $F$.

 The goal of the paper is to describe the structure of $\Gal(F(\gamma)/F)$ for the generic element of $\Gamma$. Our work was inspired by the work of Jouve, Kowalski, and Zywina \cite{JKZ}, where a similar problem is treated when $\Gamma$ is an arithmetic subgroup of a connected algebraic group $G$ defined over a number field $k$. The main result of \cite{JKZ} (which in turn is related to \cite{Gall,GN} and generalizes some special cases in \cite{Riv,KoBook,Jou,JKZ08}) is that for a typical element $\gamma$ of such a group $\Gamma$, $\Gal(k(\gamma)/k)$ is isomorphic to an explicitly described finite group $\Pi(G)$, depending only on $G$. If $G$ splits over $k$, $\Pi(G)$ is the Weyl group $W(G)$ of (the reductive part of ) $G$.

 Our goal is to generalize this result to general $F$ and $\Gamma$. As we will see, the situation can be quite different, especially in the cases where the Zariski closure $H=\overline\Gamma$ is not connected. In this case, there is {\it{no typical}} behaviour, but rather it decomposes into finitely many typical behaviours according to the connected components, but even this description does not give the full picture. To give the precise result, let us introduce some notations.

 Let $\Sigma$ be a finite admissible generating set (c.f. \cite{LM} or \S\ref{sec:sieving in groups}) of $\Gamma$. A random walk on $\Gamma$ is a map $w:\bN\to S$ where the $k$-step is $w_k=w(1)\cdots w(k)$ (so $w_0=e$ is the identity element). For a subset $Z\subset\Gamma$ we write $\dsP(w_k\in Z)$ for the probability of $w_k$ to be in $Z$.

 Let $H$ be the Zariski closure of $\Gamma$, and $H^o$ its connected component, so $H^o$ is a normal subgroup of finite index, say $m$, in $H$. We can now state our main theorem.
\begin{thm}\label{thm:main}
Let $\Gamma=\langle\Sigma\rangle\leq\Gln(F)$, where $F$ is a finitely generated field of characteristic zero, $\overline{\Gamma}=H$, and $m=(H:H^o)$ as above. Assume $H^o$ has no central tori (e.g. $H^o$ is semisimple). Then there exist $0<c,\beta\in\bR$, and a function $\Pi:H/H^o\to$Finite Groups, such that
$$
\dsP(\Gal(F(w_k)/F)\not\simeq \Pi(H^ow_k))\leq ce^{-\beta k}
$$
i.e. for any coset $H_i$ of $H^o$ in $H$, there exists a finite group $\Pi_i=\Pi(H_i)$ such that given $w_k$ is in the coset $H_i$, the probability that its associated Galois group is isomorphic to $\Pi_i$ is approaching $1$ at an exponential rate in $k$.
\end{thm}

The theorem is best possible. In \S\ref{sec:examples} we will show that the result is not necessarily true if either $F$ is not finitely generated or if $H^o$ has a central torus. As we will show there, the theorem can fail in different ways, but also in the general case one can describe how $\Gal(F(w_k)/F)$ behaves, though the description is not that enlightening.
As mentioned above, Theorem \ref{thm:main} generalizes the main theorem of \cite{JKZ} in several ways: general finitely generated field $F$, and general finitely generated group $\Gamma$. But the most interesting aspect is the fact that when $m>1$, the typical behaviour is not uniform. This can happen even for arithmetic groups. Let us illustrate this by an example:

Example: Let $\Lambda=SL_n(\bZ)$ and $\Gamma=SL_n(\bZ)\rtimes C_2$ where the cyclic group $C_2=\langle\tau\rangle$ acting on $\Lambda$ by: $\tau(A)=(^tA)^{-1}$. The group $\Gamma$ can be embedded $\rho:\Gamma\to GL_{2n}(\bZ)$ by: For $A\in SL_n(\bZ)$
$$
\rho(A)=
\begin{pmatrix}
A&0\\
0&(^tA)^{-1}
\end{pmatrix}
$$
and
$$
\rho(\tau)=
\begin{pmatrix}
0&I\\
I&0
\end{pmatrix}
$$
Now for $\gamma\in\Gamma\leq GL_{2n}(\bZ)$ we will see the following behaviour:
\begin{enumerate}
\item If $\gamma$ is in the index two subgroup $\Lambda$, then typically $\Gal(\bQ(\gamma)/\bQ)\simeq S_n$ the symmetric group of n elements (This follows from \cite{KoBook}, and in fact from \cite{Riv})
\item\label{rem:item2example} If $\gamma\in \Gamma\setminus\Lambda$ then typically $\Gal(\bQ(\gamma)/\bQ)\simeq (\bZ/2\bZ)wr_{\Omega}W_r$, where $r=\lfloor\frac n2\rfloor$, $\Omega=\{a_1,b_1,\dots,a_r,b_r\}$, and $W_r=(\bZ/2\bZ)wr S_r$, the group of signed permutations of $r$ pairs of elements, acting on $\Omega$ in the natural way.
\end{enumerate}
The fact that we get in \eqref{rem:item2example} an extension of a finite abelian
group by the Weyl group of a smaller semisimple group is not
accidental. In fact, what comes into the game is the "Weyl group of
the coset" $SL_n\cdot\tau$ (see \cite{Mohr}) which is an extension of an
abelian group by the Weyl group of the group of fixed points of
$\tau$ in $SL_n$.

Let us say a few words about the proof: By some field theoretic
argument, one can reduce the problem to the case when $F=k$ a number
field. Then $\Gamma$ is a subgroup of an arithmetic group (or more
precisely of an $S-$ arithmetic group). If $H$ is semisimple, we
will apply the recent developed sieve method for linear groups (as
in \cite{JKZ}, but we will follow \cite{LM} and in particular we
will use \cite{SGV} which gives property $\tau$ for fairly general
subgroups of arithmetic groups). As said, the main novelty of the
current paper is the treatment of the non-connected case. Here we
show that the "Weyl group of a coset" (as in \cite{Mohr}) indeed
replaces the "Weyl group of the connected component".

There is another aspect of our work which seems worth to be
mentioned here: In \cite{JKZ}, the Galois group of an element
$\gamma$ of an arithmetic group $\Gamma$ is studied over the field of definition for $\Gamma$. We study the same
problem over general fields $\dsF$. For an individual element $\gamma\in\Gamma$, the Galois group $\Gal(\dsF(\gamma)/\dsF)$ depends very much on the field $\dsF$, e.g. it is the trivial group if $\dsF$ happens to contain the eigenvalues of $\gamma$. The proof of Theorem \ref{thm:main} shows, however that for a generic $\gamma\in\Gamma$, $\dsF(\gamma)$ contains a fixed finite extension $\dsF'$ of $\dsF$. Then $\Gal(\dsF(\gamma)/\dsF')=\Gal(\dsF'(\gamma)/\dsF')$ depends only on $H=\overline\Gamma$ and neither on $\dsF$ or $\Gamma$, as long as $\dsF$ is finitely generated .The next result shows that the finite generation of $\dsF$ is crucial for generic behaviour: \setcounter{tthm}{6}
\begin{thmm}
Let $H:=SL_n$, $n\geq5$, $\Gamma:=H(\dsZ)$. Let $\Sigma$ be a finite generating set of $\Gamma$, and let $X_k$ be the corresponding random walk. Then for any pair of subgroups $\underline{\G}=(G_1,G_2)$ of the alternating group $Alt(n)$ there exists an algebraic extension $\dsF_{\G}$ of $\dsQ$, and sequences $\{n_i(\G)\},\{k_i(\G)\}$, such that
\begin{eqnarray*}
\dsP(\Gal(\dsF_{\G}(X_{n_i(\G)})/\dsF_{\G})=G_1)\geq1-\frac1{2^{i}}\\
\dsP(\Gal(\dsF_{\G}(X_{k_i(\G)})/\dsF_{\G})=G_2)\geq1-\frac1{2^{i}}
\end{eqnarray*}
\end{thmm}
\mbox{}\\
The paper is organized as follows: We will start in Section
\ref{sec:Cartans}, by defining the groups $\Pi_1,\dots,\Pi_m$
associated with the various cosets $H_1,\dots,H_m$. We further show
that when $\gamma\in H_i$, there is a map from
$\Gal(F(\gamma)/F)$ to $\Pi_i=\Pi(H_i)$. In Section
\ref{sec:reduction mod primes} we assume $\dsF$ is a finite field,
and prove some properties of Cartan subgroups and Weyl group, that
lay the background for \S\ref{sec:sieving in groups}, where we assume
that $F=k$ is a number field and $H^o$ is semisimple and prove
Theorem \ref{thm:main} in this case using the sieve method (leaving
the reduction of the general case to \S\ref{sec:general
case}). In Section \ref{sec:general case}, we will prove the
reduction from general groups and general fields to number fields.
In Section \ref{sec:examples}, we will give various examples
illustrating the various possibilities of $\Pi(H_i)$. In the last
two subsections we will explain how Theorem \ref{thm:main} can fail
if either $F$ is not finitely generated or if $H^o$ has a central
tori. We also include two appendices developing further over
\cite{Mohr,PR} some results we need on "Weyl groups of cosets" and
on $k$-quasi-irreducibility.
\subsection{Notations} Throughout the paper we will use the following notations. For a set $X$ ,we denote by $|X|$ the cardinality of the set. We write $A=O(B)$ or $A\ll B$ if there exists an absolute constant $C\geq0$ such that $A\leq CB$.  For a group $G$, we denote by $G^\sharp$ the set of conjugacy classes, and for an element $g\in G$, $[g]\in G^\sharp$ is the conjugacy class containing it. For an element $g\in G$ and a subgroup $H<G$, denote $Z_H(g)$, the centralizer of $g$ inside $H$, that is the subgroup of elements in $H$ commuting with $g$. We use $\dsF$ for an ambient field, $\overline\dsF$ its algebraic closure and $k$ for a number field. For an algebraic group $G$, $G^o$ denotes the connected component of $G$.
\subsection{Acknowledgments} We would like to thank A. Rapinchuk, M. Larsen, M. Jarden and L. Bary-Soroker for their helpful comments and discussions. We also would like to thank E. Kowalski and F. Jouve for discussions about their work.
\section{Splitting fields of elements in algebraic groups}\label{sec:Cartans}
Let $\dsF$ be a perfect field, $\Sigma\subset \Gln(\dsF)$ a finite set. Denote by $\Gamma:=\langle\Sigma\rangle$, and $H:=\overline\Gamma$ its Zariski closure. In the following section we will construct for any coset $H_i$ of $H^o$ in $H$ a finite group $\Pi(H_i)$, such that for any element $h\in H_i\cap \Gamma$ the Galois group of the splitting field of its characteristic polynomial, denoted by $\Gal(\dsF(h)/\dsF)$, is a quotient of a subgroup of $\Pi(H_i)$. To do so we recall some properties of diagonalizable groups over perfect fields.
\subsection{Diagonalizable groups}\label{sec:diag_gps}
A linear algebraic group $D$ defined over a perfect field $\dsF$ is called {\it{diagonalizable}} if there exists a faithful rational representation $\rho:D\to GL_n$ such that $\rho(D(\overline \dsF))$ is contained in the group of diagonal matrices. For such a group $D$, denote by $X(D)$ the group of characters $\chi:D\to G_m$, where $G_m$ is the one dimensional torus. The group $\Gal(\overline \dsF/\dsF)$ acts on $D$, and on $X(D)$ by $\chi^\sigma(d)=\sigma(\chi(\sigma^{-1}(d)))$, for $\sigma\in \Gal(\overline\dsF/\dsF),\chi\in X(D),d\in D$. Let $\phi_D:\Gal(\overline \dsF/\dsF)\to Aut(X(D))$ be the homomorphism such that $\phi_D(\sigma)(\chi)=\chi^\sigma$, and denote by $\dsF_D$ the fixed field of $\ker(\phi_D)$. We call this field the {\it{splitting field}} of $D$. It is a finite Galois extension of $\dsF$. The group $D$ is said to be {\it{split}} over $\dsF$ is its splitting field is $\dsF$. In fact the following conditions are equivalent
\begin{enumerate}
    \item The action of $\Gal(\overline \dsF/\dsF)$ on $X(D)$ is trivial.
    \item There exists a faithful $\dsF$-rational representation $\rho:D(\overline \dsF)\to \dsD_n(\overline \dsF)$ where $\dsD<\Gln$ is the group of diagonal matrices.
\end{enumerate}
A connected diagonalizable group is called a {\it{torus}}. In fact every diagonalizable group $D$ is of the form $D^o\times F$, where $D^o$ is a torus, and $F$ is a finite group.
\subsection{Cartan subgroups}\label{sec:Cartan_subgp}
In the theory of connected algebraic groups, the notion of a maximal torus plays a central role. But, for non-connected groups this notion is not enough. We therefore recall the following definition of Cartan subgroup (cf. \cite{Mohr}). Let $H$ be a linear algebraic group.
\begin{defn}
A Zariski closed subgroup $C<H$ is called a {\underline{\it{Cartan subgroup}}} if the following conditions hold
\begin{enumerate}
    \item $C$ is diagonalizable
    \item $C/C^o$ is cyclic.
    \item $C$ has finite index in $N_H(C)$.
\end{enumerate}
The group $W(C)=N_{H^o}(C)/C^o$ is called the outer Weyl group of $C$.
\end{defn}
Let $H_i$ be a fixed coset of $H^o$, and denote by
\begin{multline}\label{defn:Cartan_associ}
\mathcal C_i:=\left\{C<H:
C\text{ is a Cartan subgroup s.t. }C/C^o\right. \\\left.\text{is generated by }C^og\text{ with }g\in H_i\right\}
\end{multline}
We call the elements of $\mathcal C_i$ the {\it{Cartan subgroups associated to}} $H_i$.
In \cite{Mohr} some properties of Cartan subgroup are proved. We list in the proposition below some of them that are needed in this paper
\begin{prop}\label{prop:Cartan_bsc_prpts}
Let $H$ be a linear algebraic group such that $H^o$ is reductive, and let $H_i$ be a coset of $H^o$ in $H$.
\begin{enumerate}
\item
Every semisimple element $g\in H_i$ is contained in a Cartan subgroup associated to $H_i$. In fact if $T$ is a maximal torus in $\left(Z_{H^o}(g)\right)^o$, then the group $\langle T,g\rangle$ is a Cartan subgroup.
\item
If $C\in\mathcal C_i$ is a Cartan subgroup associated to $H_i$, such that $C/C^o$ is generated by $C^og$, with $g\in H_i$, then $C^o$ is a maximal torus of $\left(Z_{H^o}(g)\right)^o$.
\item Any two Cartan subgroups $C_1,C_2\in\mathcal C_i$, are conjugate by an element of $H^o$.
    \item Let $C$ be a Cartan subgroup and $h_1,h_2\in H_i\cap C$. Then $h_1$ and $h_2$ are $H^o$ conjugate if and only if they are $N_{H^o}(C)$ conjugate.
    \item If $H^o$ is semisimple and simply connected, then $Z_{H^o}(C)=C^o$, and hence, the outer Weyl group $W(H_i,C)=N_{H^o}(C)/C^o$ acts faithfully on $C$. In general $C^o$ is of finite index inside $Z_{H^o}(C)$.
\end{enumerate}
\end{prop}
\subsection{$\dsF$-Cartan subgroups}\label{sec:F-Cartans}
Since in this paper the base field plays an important role, the notion of $\dsF$-Cartan subgroup is in order. As above, let $H$ be a linear algebraic group, defined over a field $\dsF$ (in general not algebraically closed), and let $H_i$ be a coset of the connected component $H^o$. Denote by
\begin{multline*}
\mathcal C_i(\dsF)=\left\{C<H:C \text{ is a Cartan subgroup defined over }\dsF\right.\\
\left. \text{ such that }C/C^o \text{ is generated by } C^og,\; g\in H_i(\dsF)\right\}
\end{multline*}
We say that the coset $H_i$ {\it{splits over }}$\dsF$ if there exists an $\dsF$-Cartan subgroup associated to $H_i$, (i.e. a member of $\mathcal C_i(\dsF)$) that splits over $\dsF$ (note that this notion agrees with the case of connected groups, where a group is said to split over $\dsF$ if it contains an $\dsF$- split maximal torus).
Recall that for a Cartan subgroup $C\in\mathcal C_i(\dsF)$ the Galois group $\Gal(\overline\dsF/\dsF)$ acts on $X(C)$. Denote by $\dsF_C$ the splitting field of $C$. Let $C\in\mathcal C_i(\dsF)$ be a fixed $\dsF$-Cartan subgroup. The outer Weyl group $W(H_i,C)$ acts on $C$, and therefore also on $X(C)$, and can be mapped into $Aut(X(C))$. We denote its image by $\widetilde W(H_i,C)$, and denote this action by $w\cdot\chi=\chi(w(c))$ for $\chi\in X(C)$, and $c\in C$. Since $C$ is defined over $\dsF$, then its normalizer and centralizer in $H^o$, $N_{H^o}(C),Z_{H^o}(C)$, are also defined over $\dsF$, and hence the Galois group $\Gal(\overline{\dsF}/\dsF)$ acts on the outer Weyl group $W(H_i,C)$, and on $\widetilde W(H_i,C)$. Let $\sigma\in \Gal(\overline{\dsF}/\dsF)$, $w\in \W(H_i,C)$ and $\chi\in X(C)$. Then
\begin{equation}\label{eq:Gal_acts_Weyl}
\sigma(w)\cdot\chi=\left(w\cdot\chi^{\sigma^{-1}}\right)^\sigma.
\end{equation}
In particular, if $C$ splits over $\dsF$, and so the action of the Galois group is trivial on $X(C)$, it is also trivial on $\widetilde W(H_i,C)$.

We can now define the group $\Pi(H_i)$ that will contain as a subquotient the Galois group of the splitting field of the characteristic polynomial of every element in $H_i$. Let $C\in\mathcal C_i(\dsF)$ be an $\dsF$-Cartan subgroup of $H$ associated to $H_i$. Denote by $\Pi(H_i,C,\dsF)$ the subgroup of $Aut(X(C))$ generated by the image of $\widetilde W(H_i,C)$ and $\phi_C(\Gal(\dsF_C/\dsF))$. Let $C_1$ be another $\dsF$-Cartan subgroup of $H$ associated to $H_i$. Then $C_1$ and $C$ are $H^o$-conjugate, that is there exists an element $h\in H^o$ such that $C=h^{-1}C_1h$. Notice that since both Cartan subgroups are defined over $\dsF$, we have that for any $\sigma\in \Gal(\overline{\dsF}/\dsF)$ the element $h^{-1}\sigma(h)$ normalizes $C$, and therefore it defines an element of $W(H_i,C)$, and therefore an element of $\W(H_i,C^o)$. Denote this element by $w_\sigma$. Denote by $f:C_1\to C$ the isomorphism sending $x\mapsto h^{-1}xh$, and by $F:X(C)\to X(C_1)$ the one sending $\chi\mapsto\chi\circ f$.
\begin{prop}\label{prop:Pi_bsc_props} Under the above notations, we have:
\begin{enumerate}
    \item $\W(H_i,C)$ is a normal subgroup of $\Pi(H_i,C,\dsF)$.
    \item\label{prop:uniqns of Pi_ii} The isomorphism of $Aut(X(C))\to Aut(X(C_1))$ sending $\gamma\mapsto F\circ\gamma\circ F^{-1}$ induces an isomorphism of $\Pi(H_i,C,\dsF)$ onto $\Pi(H_i,C_1,\dsF)$, and of $\W(H_i,C)$ onto $\W(H_i,C_1)$.
    \item Let $\sigma\in \Gal(\overline{\dsF}/\dsF)$. Denote by $w_\sigma\in \W(H_i,C)$ the element represented by $h^{-1}\sigma(h)$. Then
    $$
    F^{-1}\circ\phi_{C_1}(\sigma)\circ F=w_\sigma\circ\phi_C(\sigma)
    $$
    \item\label{prop:Pi_bsc_props_iv} Let $K\subset\overline{\dsF}$ be an extension of $\dsF$ over which $H_i$ splits. Then $\phi_C(\Gal(\overline{\dsF}/K))\subset \W(H_i,C)$.
\end{enumerate}
\end{prop}
\begin{proof}
For $\sigma\in \Gal(\overline{\dsF}/\dsF)$ and $w\in \W(H_i,C)$ we need to show that $\phi_C(\sigma)\circ w\circ\phi_C(\sigma^{-1})$ lies in $\W(H_i,C)$. Now,
\begin{equation}
\left(\phi_C(\sigma)\circ w\circ\phi_C(\sigma^{-1})\right)(\chi)=\left(w\cdot\chi^{\sigma^{-1}}\right)^\sigma=\sigma(w)\cdot\chi
\end{equation}
We therefore get that $\phi_C(\sigma)\circ w\circ\phi_C(\sigma^{-1})=\sigma(w)\in \W(H_i,C)$.
For $\sigma\in \Gal(\overline{\dsF}/\dsF)$ we will show that $F^{-1}\circ\phi_{C_1}(\sigma)\circ F$ lies in $\Pi(H_i,C,\dsF)$.
\begin{equation}
\left(F^{-1}\circ\phi_{C_1}(\sigma)\circ F\right)(\chi)=(\chi\circ f)^\sigma\circ f^{-1}=\chi^\sigma\circ(f^\sigma\circ f^{-1})=\chi^\sigma\circ(f\circ (f^\sigma)^{-1})^{-1}
\label{eq:pi(Hi,C) isomorphism}
\end{equation}
where the isomorphism $f\circ (f^\sigma)^{-1}$ of $C$ maps $x\mapsto h^{-1}\sigma(h)x(h^{-1}\sigma(h))^{-1}$, which equals the isomorphism of $C$ given by $w_\sigma$. This proves that
\begin{equation}\label{eq:galois lmnts in Pi}
    F^{-1}\circ\phi_{C_1}(\sigma)\circ F=w_\sigma\circ\phi_C(\sigma)
\end{equation}
and is therefore an element of $\Pi(H_i,C,\dsF)$, and we also proved (3). To finalize the proof of part (2) we need to see the isomorphism of the Weyl groups: we note that if $w\in W(H_i,C_1)$ is represented by $n\in N_{H^o}(C)$ then $h^{-1}nh\in N_{H^o}(C_1)$. The isomorphism of $\W(H_i,C),\W(H_i,C_1)$ then follows immediately.

For the last part, if $H_i$ splits over $K$, we may then take $C$ as a split Cartan group. We therefore get by \eqref{eq:galois lmnts in Pi} that
$$
    \phi_{C_1}(\sigma)=F\circ w_\sigma\circ F^{-1}
$$
which is an element of $\W(H_i,C_1)$ by part \eqref{prop:uniqns of Pi_ii} of the proposition.
\end{proof}
For a fixed field $\dsF$, the group $\Pi(H_i,C,\dsF)$ is unique up to isomorphism, and the isomorphisms $\Pi(H_i,C,\dsF)\widetilde{\to}\Pi(H_i,C',\dsF)$ are unique up to inner automorphisms of $H^o$. We can therefore denote it by $\Pi(H_i,\dsF)$, and the set of conjugacy classes $\Pi(H_i,C,\dsF)^\sharp$ are unambiguous and we will denote it by $\Pi(H_i,\dsF)^\sharp$. Notice that by the last part of Proposition \ref{prop:Pi_bsc_props} we have that if $H_i$ splits over $\dsF$, then $\Pi(H_i,C,\dsF)=\W(H_i,C)$.
 In the case that two or more fields of different types take place we denote the set of conjugacy classes by $\W(H_{i,\overline\dsF})^\sharp$.

 For a coset $H_i$, define the {\it{splitting field of}} $H_i$ to be $\dsF_i:=\cap_{C\in\mathcal C_i(\dsF)}\dsF_C$, and $\dsF_i^W=\phi_C^{-1}(\W(H_i,C))$ for an $\dsF$-Cartan subgroup $C$ associated to $H_i$ (notice that the definition of $\dsF_i^W$ is independent of $C$ by Proposition \ref{prop:Pi_bsc_props}\eqref{prop:uniqns of Pi_ii}). Although $H_i$ does not necessarily split over $\dsF_i$, the following lemma shows that some connection with the splitting property still holds.
\begin{lem}\label{lem:splitng fld}
Let $\dsF$, $H$, and $H_i$ be as above. Then
\begin{enumerate}
\item\label{lem:splitng fld1}
For any $\dsF$-Cartan subgroup $C$ associated to $H_i$,
$\phi_C(\Gal(\overline{\dsF}/\dsF_i))\subseteq \W(H_i,C)$. i.e. $\dsF_i^W\subset\dsF_i$, and in particular $\Gal(\overline{\dsF}/\dsF_i)\leq\Gal(\overline{\dsF}/\dsF_i^W)$
\item
For any $\dsF$-Cartan subgroup $C$ associated to $H_i$, the following holds:
\begin{enumerate}
\item
$\Gal(\dsF_C/\dsF_i^W)$ is isomorphic to a subgroup of $\W(H_i,C)$.
\item
There is an exact sequence:
\begin{equation}\label{eq:Pi image seq}
1\to \widetilde W(H_i,C)\cap \phi_C(\Gal(\overline{\dsF}/\dsF))\to\phi_C(\Gal(\overline{\dsF}/\dsF))\to \Gal(\dsF_i^W/\dsF)\to1.
\end{equation}
Furthermore, the group $\Pi(H_i,C,\dsF)$ satisfies the following exact sequence
\begin{equation}
1\to \widetilde W(H_i,C)\to\Pi(H_i,C,\dsF)\to \Gal(\dsF_i^W/\dsF)\to1
\label{eq:Pi exct seq}
\end{equation}
\end{enumerate}
\item{\bf{Functoriality of $\Pi$}}:
For any finite field extension $E/\dsF$, $\Pi(H_i,C,E)\leq\Pi(H_i,C,\dsF)$.
\item\label{prop:functorPi}
Let $C$ be an $\dsF$-Cartan subgroup associated to $H_i$. Then if there exists a finite extension $E$ such that
 $\phi_C:\Gal(\overline{\dsF}/E)\to\Pi(H_i,C,E)$ is surjective, then $\phi_C:\Gal(\overline{\dsF}/\dsF)\to\Pi(H_i,C,\dsF)$ is also surjective
\end{enumerate}
\end{lem}
\begin{remark}
\begin{enumerate}
\item
We do not know if $\dsF_i^W=\dsF_i$. In \cite{PR2}, Prasad and Rapinchuk show that if $H$ is connected, then $\dsF_i^W$ is the smallest field $L$ such that $H$ is an inner form over $L$.
\item
Part \eqref{prop:functorPi} will allow us, in the proof of Theorem \ref{thm:main}, to assume that $H_i$ splits over $\dsF$
\end{enumerate}
\end{remark}
\begin{proof}\mbox{}
\begin{enumerate}
\item
Since $\W(H_i,C)$ is a normal subgroup of $\Pi(H_i,C,\dsF)$, then so is its inverse image $W_\phi:=\phi^{-1}(\W(H_i,C))$. Furthermore, as mentioned above, by Proposition \ref{prop:Pi_bsc_props}\eqref{prop:uniqns of Pi_ii}, it is independent of $C$. $\dsF_i\subset\overline{\dsF}$ is the fixed field of $W_\phi$, which is the finite extension of $\dsF$ satisfying that it is the minimal extension of $\dsF$ such that $\phi_C(\Gal(\overline{\dsF}/L))\subset \W(H_i,C)$. For a Cartan subgroup $C_1$, by Proposition \ref{prop:Pi_bsc_props}\eqref{prop:Pi_bsc_props_iv} we have that $\dsF_i^W\subset \dsF_{C_1}$. Since $C_1$ is arbitrary, we get that $\dsF_i^W\subset\dsF_i$.
\item
\begin{enumerate}
\item
As $\ker(\phi_C)=\Gal(\overline{\dsF}/\dsF_C)$, we get by the first part that $\Gal(\dsF_C/\dsF_i)$ is isomorphic to a subgroup of $\W(H_i,C)$.
\item
By \eqref{lem:splitng fld1}, $\dsF_i^W\subset\dsF_C$, and therefore $\Gal(\overline{\dsF}/\dsF_C)\leq \Gal(\overline{\dsF}/\dsF_i^W)$. Also by the definition of $\dsF_i^W$, and the fact that $\ker(\phi_C)=\Gal(\overline{\dsF}/\dsF_C)$, we get that
\begin{multline}
\phi_C(\Gal(\overline{\dsF}/\dsF))/\left(\W(H_i,C)\cap\phi_C(\Gal(\overline{\dsF}/\dsF))\right)\simeq \\ \left(\Gal(\overline{\dsF}/\dsF)/\Gal(\overline{\dsF}/\dsF_C)\right)/\left(\Gal(\overline{\dsF}/\dsF_i^W)/\Gal(\overline{\dsF}/\dsF_C)\right)\simeq \Gal(\dsF_i^W/\dsF)
\end{multline}
By definition, $\Pi(H_i,C,\dsF)$ is generated by $\phi_C(\Gal(\overline{\dsF}/\dsF))$ and $\W(H_i,C)$. Since $\W(H_i,C)$ is a normal subgroup of $\Pi(H_i,\dsF,C)$, it is in fact equal to $\W(H_i,C)\phi_C(\Gal(\overline{\dsF}/\dsF))$. Therefore, by the computation above, we have that
\begin{multline}
\Pi(H_i,\dsF,C)/\W(H_i,C)=\W(H_i,C)\phi_C(\Gal(\overline{\dsF}/\dsF))/\W(H_i,C)\simeq\\
\phi_C(\Gal(\overline{\dsF}/\dsF))/\left(\W(H_i,C)\cap\phi_C(\Gal(\overline{\dsF}/\dsF))\right)\simeq \Gal(\dsF_i^W/\dsF)
\end{multline}
\end{enumerate}
\item
Since $\phi_C(\Gal(\overline{\dsF}/E))$ is a subgroup of $\phi_C(\Gal(\overline{\dsF}/\dsF))$, and by the definition of the groups, we get that $\Pi(H_i,C,\dsF)\leq\Pi(H_i,C,E).$
\item
By \eqref{eq:Pi image seq},\eqref{eq:Pi exct seq}, we get that for any field $K$, $\phi_C:\Gal(\overline{\dsF}/K)\to\Pi(H_i,C,K)$ is surjective if and only if the image contains $\W(H_i,C)$. Therefore, if there exists a field E such that $\phi_C:\Gal(\overline{\dsF}/E)\to\Pi(H_i,C,E)$ is surjective, then the image $\phi_C(\Gal(\overline{\dsF}/\dsF)$ contains $\W(H_i,C)$, and therefore, $\phi_C:\Gal(\overline{\dsF}/\dsF)\to\Pi(H_i,C,\dsF)$ is surjective.
\end{enumerate}
\end{proof}
Let $\sigma\in \Gal(\overline{\dsF}/\dsF)$. We define a function $\theta_\sigma:\mathcal C_i(\dsF)\to\Pi(H_i,\dsF)^\sharp$ given by $\theta_\sigma(C)=[\phi_C(\sigma)]$. Notice that if $H_i$ splits over $\dsF$, then the image is in fact inside $\W(H_i)^\sharp$. In section \ref{appsec:reg ss} we define the notion of regular semisimple elements. A semisimple element is called {\it{regular}} if the connected component of its centralizer in $H^o$ is a torus. From Proposition \ref{prop:Cartan_basic_props} (2) it follows that such an element $g\in H_i$ is contained in a unique Cartan subgroup associated to $H_i$ that we denote $C_g$. Note that when $g\in H_i(\dsF)$, then $C_g$ is defined over $\dsF$. Using this we can extend $\theta_\sigma$, defined above for Cartan subgroups, to the set of regular semisimple elements in $H_i(\dsF)$, which we denote $\left(H_i(\dsF)\right)_{sr}$. We do this by defining for an element $\sigma\in \Gal(\overline{\dsF})$ the map $\theta_\sigma:\left(H_i(\dsF)\right)_{sr}\to \Pi(H_i,\dsF)^\sharp $ with $\theta_\sigma(g)=\left[\phi_{C_g}(\sigma)\right]$.

The following proposition shows that in the case where $\chara(\dsF)=0$, for any element $h\in H_i(\dsF)$, the Galois group of the splitting field over $\dsF$ of the characteristic polynomial of $h$ is a subquotient of $\Pi(H_i,\dsF)$.
\begin{prop}\label{prop:Gal subqtint Pi}
Let $H,H_i,\dsF_i$ be as above, and assume $\chara(\dsF)=0$. For any $h\in H_i(\dsF)$, let $\Gal(\dsF(h)/\dsF)$ be the Galois group of the splitting field of $\det(T-h)$. Then $\Gal(\dsF(h)/\dsF)$ is isomorphic to a quotient of a subgroup of $\Pi(H_i,\dsF)$, and
 $\Gal(\dsF(h)/(\dsF(h)\cap \dsF_i))$ is isomorphic to a quotient of a subgroup of $\W(H_i)$.
\end{prop}
\begin{proof}
Let $h\in H_i(\dsF)$. By Jordan decomposition there exists a unique pair $h_s,h_u\in H(\dsF)$ with $h_s$ semisimple, and $h_u$ unipotent, such that $h=h_sh_u=h_uh_s$. Furthermore, since $h_u$ is unipotent, the group generated by it is connected and therefore $h_u\in H^o$, and so $h_s\in H_i(\dsF)$. Let $D_h:=D_{h_s}$ be the Zariski closed subgroup generated by $h_s$. Then $D_h$ is contained in any Cartan subgroup $C$ of $G$ containing $h_s$. Let $C_h$ be a Cartan subgroup containing $h_s$, such that $C_h/C^o_h$ is generated by $C^o_hh_s$, so $C_h\in\mathcal C_i(\dsF)$. The splitting field of $\det(T-h)=det(T-h_s)$ is the splitting field of $D_h$. Since $D_h\subset C_h$, we have that the splitting field of $D_h$ is contained in the splitting field of $C_h$. By the construction of $\Pi(H_i,C_h,\dsF)$ we have that $\Gal(\dsF_{C_h}/\dsF)$ is a subgroup of $\Pi(H_i,C_h,\dsF)$, and since $\Gal(\dsF_{D_h}/\dsF)$ is a quotient of $\Gal(\dsF_{C_h}/\dsF)$ the result follows. Since $\dsF_i\subset\dsF_{C_h}$, and $\Gal(\dsF_{C_h}/\dsF_i)\subset \W(H_i,C_h)$ (by lemma \ref{lem:splitng fld}) the second part follows in a similar way.
\end{proof}
\subsection{Outer Weyl groups and reduction modulo primes}
In this section we apply the above construction in the following setting: Let $\dsF=k$ be a number field, $\Sigma\subset\Gln(k)$ a finite set, $\Gamma=\langle\Sigma\rangle$, $H=\overline\Gamma$. Let $\mathcal O_k$ be the ring of integers of $k$. There exists a finite set of places $S$, such that $\Gamma\subset\Gln(\mathcal O_{k,S})\cap H=:H(\mathcal O_{k,S})$. A significant role in this paper is played by the Frobenius conjugacy class. We give here a reminder for that. Let $\ideal p\triangleleft\mathcal O_k$ be an unramified prime ideal, and let $k_{\ideal p},\mathcal O_{\ideal p}$ be the corresponding field completion, and valuation ring, respectively. Let $\dsF_{\ideal p}$ be the residue field, and $\pi_{\ideal p}:\mathcal O_{\ideal p}\to\dsF_{\ideal p}$ be the residue map. Denote by $k^{nr}_{\ideal p}\subset\ck_{\ideal p}$ the maximal unramified extension, and the algebraic closure of $k_{\ideal p}$ of $k_{\ideal p}$, respectively. It is a well known fact (c.f. chapter V in \cite{CaF}) that $\Gal(k^{nr}_{\ideal p}/k)\simeq \Gal(\overline{\dsF_{\ideal p}}/\dsF_{\ideal p})$ are isomorphic, and thus there exists a unique element denoted $\Frob_{\ideal p}$ that corresponds to the generator of $\Gal(\overline{\dsF_{\ideal p}}/\dsF_{\ideal p})$ (denoted also by $\Frob_{\ideal p}$ sending $x\mapsto x^{N(\ideal p)}$, where $N(\ideal p)$ is the cardinality of $\dsF_{\ideal p}$). For an embedding $\ck\to\ck_{\ideal p}$ corresponds an inclusion $\Gal(\ck_{\ideal p}/k_{\ideal p})\to \Gal(\ck/k)$, which defines an element $\Frob_{\ideal p}\in \Gal(\ck/k)$. This element is defined up to a conjugation, and thus defines a conjugacy class of $\Gal(\ck/k)$, and, as in \S\ref{sec:F-Cartans}, defines a map from the regular semisimple elements of $H_i(k)$ to conjugacy classes of $\Pi(H_i,k)$. That is, the following map
$$
\theta_{\Frob_{\ideal p}}:\left(H_i(k)\right)_{sr}\to \Pi(H_i,k)^\sharp
$$
is well defined.

Let us now consider the reduction modulo prime ideals. Note that since $H$ is defined over $k$, by a finite number of polynomials, $H$ and therefore also $H^o$, and any coset of it, are defined over $\dsF_{\ideal p}$ for almost all prime ideals $\ideal p$ (depending only on $H$), by using the same polynomials defining $H$, regraded as polynomials over $\dsF_{\ideal p}$. Notice also, that given a coset $H_i$, the Weyl group $W(H_i)$, seen as $N_{H^o}(C)/C^o$ for a $k$-Cartan subgroup $C$ is also defined over $\dsF_{\ideal p}$ for all large enough $\ideal p$, and in fact as it is finite, it is isomorphic to the Weyl group defined over $k$.
The goal of this subsection is to prove the following proposition:
\begin{prop}\label{prop:reduction mod primes}
Let $k,\mathcal O_k,S, H,H^o,H_i$ be as above, and assume that $H_i$ splits over $k$. Then there exists a finite set of primes $S'$ containing $S$ such that
\begin{enumerate}
    \item There exists a non-empty open subset $V\subset \left(H_i(\ck)\right)_{rs}$ defined over $k$, such that for any $\ideal p\not\in S'$, if $h\in H_i(\mathcal O_{k,S})$ satisfying $\pi_{\ideal p}(h)\in V(\dsF_{\ideal p})$, then $h\in V(\mathcal O_{k,S})\subset\left(H_i(\mathcal O_{k,S})\right)_{rs}$, that is
\begin{multline*}
    \left\{h\in H_i(\mathcal O_{k,S}):h\not\in V(\mathcal O_{k,S})\right\}\subset\\
\left\{h\in H_i(\mathcal O_{k,S}):\pi_{\ideal p}(h)\not\in V(\dsF_{\ideal p})\forall\ideal p\not\in S'\right\}
\end{multline*}
    \item For any $\ideal p\not\in S'$ there exists a bijection $\alpha:\W(H_{i,\ck})^\sharp\to \W(H_{i,\overline{\dsF_{\ideal p}}})^\sharp$, such that for any $h\in H_i(\mathcal O_{k,S})$ if $\pi_{\ideal p}(h)\in\left(H_i(\dsF_{\ideal p})\right)_{sr}$, then $\alpha(\theta_{\Frob_{\ideal p}}(h))=\theta_{\Frob_{\ideal p}}(\pi_{\ideal p}(h))$. In particular, we have that the following diagram commutes:
    $$
\xymatrix{
H_i(\mathcal O_{k,S})_{rs}\supset V(\mathcal O_{k,S}) \ar@<5ex>[d]_{\pi_{\ideal p}} \ar[r]^(.59){\theta_{\Frob_{\ideal p}}}
& \W(H_i,\ck) \ar[d]^\alpha \\
H_i(\dsF_{\ideal p})_{rs}\supset V(\dsF_{\ideal p}) \ar[r]_(.59){\theta_{\Frob_{\ideal p}}} & \W(H_i,\overline{\dsF_{\ideal p}}) }
    $$
    and for any $U\in \W(H_i)^\sharp$
    \begin{multline*}
     \left\{h\in H_i(\mathcal O_{k,S}):h\in\left(H_i(\mathcal O_{k,S})\right)_{sr}{\text{ and }}\theta_{\Frob_{\ideal p}}(h)\ne U\;\forall\ideal p\not\in S'\right\}\subset\\\left\{h\in H_i(\mathcal O_{k,S}):\pi_{\ideal p}(h)\in\left(H_i(\dsF_{\ideal p})\right)_{sr}\wedge\theta_{\Frob_{\ideal p}}(\pi_{\ideal p}(h))\ne \alpha(U)\;\forall\ideal p\not\in S'\right\}
    \end{multline*}
\end{enumerate}
\begin{remark}\label{rk:H_i splits over all flds}
Notice that since $H_i$ is assumed to be split, there exists a Cartan subgroup $C$ associated to $H_i$ that splits, and therefore for all but finitely many primes, $H_i$ splits also over $\dsF_{\ideal p}$. In particular $\theta_{\Frob_{\ideal p}}\in\widetilde W(H_i)^\sharp$. Throughout the proof we will assume that the set $S'$ contains the finite set of exceptional primes.
\end{remark}
\begin{proof}
Let $V\subset H_i$ be the open subset given by Proposition \ref{prop:most lmnts rs}. Let $C$ be a fixed split Cartan subgroup, and $S'$ be such that $C$ splits for all $\ideal p\not\in S'$. Furthermore, let $A(C^o)$ be the set of roots of $C^o$. Since $C$ splits, also $C^o$ does, and choose $S'$ such that for all $\ideal p\not\in S'$ the elements of $A(C^o)$ are defined over $\dsF_{\ideal p}$, and are all different. Notice also that if an element $h$ is inside the kernel of a root of $C^o$, then its reduction modulo any prime will also be inside the kernel. Therefore, by the proof of Proposition \ref{prop:most lmnts rs} this shows that if $\pi_{\ideal p}(h)\in V(\overline{\dsF_{\ideal p}})$ then $\pi_{\ideal p}(h)$ is regular semisimple, and part (1) is proved.
For the second part, we follow Lemma 3.2 in \cite{JKZ}. We first construct the bijection $\alpha:\W(H_{i,\ck})^\sharp\to \W(H_{i,\overline{\dsF_{\ideal p}}})^\sharp$. Let $C_0<H$ be a fixed split $k$-Cartan subgroup associated to $H_i$. Let $\ideal p\not\in S$ be an unramified prime. The following maps
\begin{eqnarray*}
\left(N_{H^o}(C_0)/Z_{H^o}(C^o_0)\right)(\left(\mathcal O_{k,S}\right)_{\ideal p})\hookrightarrow \left(N_{H^o}(C_0)/Z_{H^o}(C^o_0)\right)(k_{\ideal p})=\W(H_{i,\ck},C_0)\\
\left(N_{H^o}(C_0)/Z_{H^o}(C^o_0)\right)(\left(\mathcal O_{k,S}\right)_{\ideal p})\twoheadrightarrow
\left(N_{H^o}(C_0)/Z_{H^o}(C^o_0)\right)(\dsF_{\ideal p})=\W(H_{i,\dsF_{\ideal p}},C_0)
\end{eqnarray*}
are indeed injective and surjective (notice that the latter is surjective by Hensel's Lemma, and the fact that $C_0$ splits), and since $\W(H_{i,\ck},C_0)\simeq \W(H_{i,\overline{\dsF_{\ideal p}}},C_0)$ are isomorphic (recall that these groups are isomorphic for $\ideal p$ outside a finite set by the discussion before the Proposition). Therefore, they are both isomorphisms, and as such define a bijection between conjugacy classes. Let $h\in H_i(\mathcal O_{k,S})$ be a regular semisimple element such that $\pi_{\ideal p}(h)\in H_i(\dsF_{\ideal p})$ is regular semisimple, $C_h=\langle \left(Z_{H^o}(h)\right)^o,h\rangle$ be the unique Cartan containing it, and $C_{h,\ideal p}$ be the unique Cartan containing $\pi_{\ideal p}(h)$. Let $x\in H^o(\overline{\dsF})$ be such that $C_0=xC_{h,\ideal p}x^{-1}$. Then by Proposition \ref{prop:Pi_bsc_props}\eqref{prop:Pi_bsc_props_iv} $\theta_{\Frob_{\ideal p}}(\pi_{\ideal p}(h))=[x^{-1}Frob_{\ideal p}(x)]$. By Hensel's Lemma applied to $\left(\mathcal O_{k,S}\right)_{\ideal p}^{nr}\to \overline{\dsF_{\ideal p}}$, there exists $\overline x\in N_{H^o}(H)(\left(\mathcal O_{k,S}\right)_{\ideal p}^{nr})$ that lifts $x$, and thus satisfy $\overline x^{-1}C_0 \overline x=C$, and thus $\theta_{\Frob_{\ideal p}}(h)=[\overline x^{-1}Frob_{\ideal p}(\overline x)]$. By the definition of the bijection of conjugacy classes, the result is proved.
\end{proof}
\end{prop}
\section{Outer Weyl groups over finite fields}\label{sec:reduction mod primes}
In this section we prove some properties of outer Weyl groups for groups defined over finite fields. Here $H$ is a linear algebraic group defined and split over a finite field $k=\dsF_q$ of $q$ elements, $H^o$, its unit component, is semisimple, and we fix a connected component $H_i$. We recall first the definition of the function $\theta$ defined in the preceding section. For a Cartan subgroup $C$, we have defined in Section \ref{sec:Cartans} a homomorphism
$$
\phi_C:\Gal(\ck/k)\to \W(H_i,C)
$$
(note that the image is indeed inside $\W(H_i)$ since we assume that $H$ splits over $k$). Let $h\in H_i$ be a semisimple regular element, and let $C_h$ be the unique Cartan subgroup containing it. Denote by $\Frob$ the Frobenius element in $\Gal(\ck/k)$. We defined in Section \ref{sec:Cartans} the following functions
\begin{align}\label{def:theta function4Cartan}
\theta_{\Frob}:\mathcal C_i(k)&\to \W(H_i)^\sharp\\
\nonumber C &\mapsto[\phi_{C}(\Frob)]
\end{align}
and
\begin{align}\label{def:theta function}
\theta_{\Frob}:H_i(k)_{sr}&\to \W(H_i)^\sharp\\
\nonumber h&\mapsto[\phi_{C_h}(\Frob)]
\end{align}
We abbreviate $\theta:=\theta_{\Frob}$. Our goal in this section is to prove that for a fixed connected component $H_i$, and for a fixed conjugacy class $U\in \W(H_i)^\sharp$, the density of elements $h\in H_i(k)$ such that $\theta(h)=U$ is positive as the size of the field $k$ grows. For a connected component $H_i$, let $\mathcal C_i$ and $\mathcal C_i(k)$ be as before.
The following theorem, known as the Lang-Steinberg Theorem ( c.f. \cite{St} Theorem 10.1), plays a crucial role in the sequel, and we therefore state it.
\begin{thm}[Lang-Steinberg Theorem]\label{thm:Lang Steinberg}
Let $G$ be a connected reductive linear algebraic group, and let $F$ be an endomorphism of $G$. Suppose that $F$ has only finitely many fixed points. Then the map $x\mapsto x^{-1}F(x)$ is surjective.
\end{thm}
Throughout this section, denote for any Cartan subgroup $C$, $\pi:N_{H^o}(C)\to W(H_i,C)$, and $\widetilde\pi:N_{H^o}(C)\to \widetilde W(H_i,C)$.
\begin{lem}\label{lem:fin fld conj corsp}
Let $H$ and $H_i$ be as above. Let $\widetilde C$ be a split $k$-Cartan subgroup associated to $H_i$, and denote by $W^F=W^F(H_i,\widetilde C):=\{w\in W(H_i,\widetilde C):\Frob(w)=w\}$. Then
\begin{enumerate}
\item
Every $C\in\mathcal C_i(k)$ defines a $\W(H_i,\widetilde C)$ conjugacy class, given by $\theta(C)$.
\item
For every $\W(H_i,\widetilde C)$-conjugacy class $U$ there exists a Cartan subgroup $C\in\mathcal C_i(k)$ such that $\theta(C)=U$.
\item
Let $C_1,C_2\in\mathcal C_i(k)$, and let $x_i,i=1,2$ be such that $C_i=\widetilde C^{x_i},i=1,2$. Assume that $\pi(x_1^{-1}\Frob(x_1))$ is $W^F$-conjugate to $\pi(x_2^{-1}\Frob(x_2))$ then $C_1$ is conjugate to $C_2$ by an element of $H^o(k)$. Conversely, if $C_1$ and $C_2$ are conjugated by an element of $H^o(k)$ then $\theta(C_1)=\theta(C_2)$
\end{enumerate}
\end{lem}
\begin{remark}
Note that as $H_i$ splits, $\W(H_i,\widetilde C)=\W(H_i,\widetilde C)^F$, but we do not know if also $W(H_i,\widetilde C)=W(H_i,\widetilde C)^F$.
\end{remark}
\begin{proof}
 For $C\in\mathcal C_i(k)$, there exists $x\in H^o(\ck)$ such that $C^x=\widetilde{C}$, and $\theta(C)=[x^{-1}\Frob(x)]\in W(H_i,\widetilde C)^\sharp$. If $\widetilde C=C^y$ for $y\in H^o(\ck)$, we have that $x^{-1}y=n\in N_{H^o}(\widetilde C)$, and
$$
y^{-1}\Frob(y)=n^{-1}x^{-1}\Frob(x)\Frob(n)=n^{-1}x^{-1}\Frob(x)n\pmod{Z_{H^o}(\widetilde C)}
$$
because $n=\Frob(n)\pmod {Z_{H^o}(\widetilde C)}$ since $\widetilde C$ splits. Therefore we get a corresponding equality in $\W(H_i,\widetilde C)$. Let $U\subset\W(H_i,\widetilde C)$ be a conjugacy class, and let $n\in N_{H^o}(C)$ a corresponding element. Then by Lang-Steinberg Theorem there exists $g\in H^o(\ck)$ such that $g^{-1}\Frob(g)=n$. Then $\widetilde C^{g^{-1}}$ is a $k$-Cartan subgroup such that $\theta(\widetilde C^{g^{-1}})=U$.

Assume now that $C_1,C_2\in\mathcal C_i(k)$ are such that $C_1^{x_1}=\widetilde C=C_2^{x_2}$, and $\pi(x_1^{-1}\Frob(x_1))$ and $\pi(x_2^{-1}\Frob(x_2))$ are conjugated by an element $W^F$. Let $n\in N_{H^o}(\widetilde C)$ be a conjugating element, that is
\begin{equation}
n^{-1}x_1^{-1}\Frob(x_1)n=x_2^{-1}\Frob(x_2)\pmod{\widetilde{C}^o}
\label{eq:conj prf}
\end{equation}
Since $w\in W^F$, we get that $n=\Frob(n)\pmod{\widetilde{C}^o}$, and so from \eqref{eq:conj prf} we get that
$$
x_2n^{-1}x_1^{-1}\Frob(x_1nx_2^{-1})\in (\widetilde{C}^o)^{x_2^{-1}}=C^o_2
$$
Applying the Lang-Steinberg theorem for $C^o_2,k$ and $\Frob$, we find that there exists $g\in C^o_2$ such that
\begin{eqnarray}
x_2n^{-1}x_1^{-1}\Frob(x_1nx_2^{-1})=g^{-1}\Frob(g)\\
gx_2n^{-1}x_1^{-1}=\Frob(gx_2n^{-1}x_1^{-1})
\label{eq:}
\end{eqnarray}
hence $y=gx_2n^{-1}x_1^{-1}\in H^o(k)$, and it satisfies $C_1^y=C_2$. The converse direction is clear.
\end{proof}
\begin{lem}\label{lem:Weyl centralizers}
Let $H,H^o$ and $H_i$ be as above. Let $\widetilde C\in\mathcal C_i(k)$ be a split Cartan subgroup, and let $C\in\mathcal C_i(k)$. Let $g\in H^o(\ck)$ be such that $C=\widetilde C^g$, and denote by  $w=\widetilde\pi(g^{-1}\Frob(g))\in \W(H_i,\widetilde C)$. Let $N(k)=N_{H^o}(C)(k), C^o(k)$ be the $k$-points of $N_{H^o}(C), C^o$ respectively. Then
\begin{enumerate}
\item Let $\pi_W:W(H_i,\widetilde C)\twoheadrightarrow\W(H_i,\widetilde C)$, then for any $x\in W^F(H_i,C)$
$$
\pi(gxg^{-1})\in Z_{\W(H_i,\widetilde C)}(w)
$$
\item Let $A=|\ker(\pi_W)|$. Then
$$
\frac{|N(k)|}{|C^o(k)|}\leq A|Z_{\W(H_i,\widetilde C)}(w)|
$$
In particular, there exists a constant $c>$ independent of $k$, such that $\frac{|N(k)|}{|C^o(k)|}<c$.
\end{enumerate}
\end{lem}
\begin{proof}
We first note that conjugation by $g$ sends $N_{H^o}(\widetilde C)$ to $N_{H^o}(C)$, and $\widetilde C^o$ to $C^o$. Let $x\in N_{H^o}(\widetilde C)$, and let $\Frob$ be the Frobenius automorphism of $k$. Then $gxg^{-1}\in W(H_i,\widetilde C)$. Notice that
$$
\Frob(gxg^{-1})=gg^{-1}\Frob(g)\Frob(x)\Frob(g)^{-1}gg^{-1}=gw^{-1}\Frob(x)wg^{-1}
$$
where all equalities are as elements of $W(H_i,\widetilde C)$. Therefore, if $gxg^{-1}\in W^F(H_i,\widetilde C)$, then
$$
gxg^{-1}=gw^{-1}\Frob(x)wg^{-1}\Rightarrow x=w^{-1}\Frob(x)w
$$
since $\widetilde C$ splits which implies that $\pi_W(x)=\pi_W(\Frob(x))$. We get that $w^{-1}\pi_W(x)w=\pi_W(x)$.

 By \cite{Ca} \S1.17 we have that $N(k)/C^o(k)$ is isomorphic to $(N/C^o)(k)$, the subgroup of $W(H_i,C)$ consisting of Frobenius fixed points, that is $W^F(H_i,C)$, which concludes the proof of the second part.
\end{proof}
\begin{prop}\label{prop:Conjugacy_density}
Let $H,H^o$ and $H_i$ be as above, and let $U\in W(H_i)^\sharp$ a fixed conjugacy class. Then there exists a constant $c>0$ (independent of $k$), such that
\begin{equation}\label{eq:conjugacy_density}
\frac{|\left\{g\in\left(H_i(k)\right)_{sr}:\theta(g)\in U\right\}|}{|H_i(k)|}\geq c>0\quad\text{as }|k|\to\infty
\end{equation}
\end{prop}
\begin{proof}
Denote by $\mathcal C_i(k,U)$ the set of Cartan subgroups $C\in\mathcal C_i(k)$ such that $\theta(C)=U$. This is a nonempty set by Lemma \ref{lem:fin fld conj corsp}. Let $C_1\in\mathcal C_i(k,U)$ be a fixed Cartan subgroup. Then
\begin{eqnarray*}
\frac{\left\{g\in\left(H_i(k)\right)_{sr}:\theta(g)\in U\right\}}{|H_i(k)|}=\frac{|\left\{g\in H_i(k)_{sr}: \exists C\in\mathcal C_i(k,U)\mbox{ s.t } g\in C\right\}|}{|H_i(k)|}\geq\\
\frac{|\left\{g\in H_i(k)_{sr}: \exists x\in H^o(k)\mbox{ s.t } g\in C_1^x\right\}|}{|H_i(k)|}=
\frac1{|H_i(k)|}\sum_{\substack{C=C_1^x\\x\in H^o(k)}}|C\cap H_i(k)|+o(1)
\end{eqnarray*}
Where in the last line we use the fact that most elements are semisimple regular, and the fact that if $C, C'$ are $H^o(k)$-conjugate then $\theta(C)=\theta(C')$. We now get that
\begin{eqnarray*}
\frac1{|H_i(k)|}\sum_{\substack{C=C_1^x\\x\in H^o(k)}}|C\cap H_i(k)|=\frac1{|H_i(k)|}\sum_{g\in H^o(k)/N_{H^o}(C_1)(k)}|g^{-1}Cg\cap H_i(k)|=\\
\frac1{|H_i(k)|}\sum_{g\in H^o(k)/N_{H^o}(C_1)(k)}|g^{-1}(C_1\cap H_i(k))g|=\frac{|C_1\cap H_i(k)|}{|N_{H^o}(C_1)(k)|}
\end{eqnarray*}
 Let $h\in C_1\cap H_i(k)$ be a fixed element, and write $H_i(k)=H^o(k)h$. Therefore
$$
|C_1\cap H_i(k)| =|C_1\cap H^o(k)h|=|C_1\cap H^o(k)|\geq|C_1^o(k)|
$$
since $C_1^o\subset\left(C_1\cap H^o\right)$.
Therefore,
\begin{eqnarray*}
\frac{|C_1\cap H_i(k)|}{|N_{H^o}(C_1)(k)|}\geq\frac{|C_1^o(k)|}{|N_{H^o}(C_1)(k)|}\geq\\
\frac{A}{|Z_{W(H_i,C_1)}(w)|}=A\frac{|U|}{|\W(H_i)|}>0
\end{eqnarray*}
where in the first equality we use $|H^o(k)|=|H_i(k)|$, and in the last line we use Lemma \ref{lem:Weyl centralizers} and that the size of a conjugacy class $U$ in $\W(H_i,C_1)$ is $|U|=\frac{|\W(H_i,C_1)|}{|Z_{\W(H_i,C_1)}(w)|}$, for any $w\in U$.
\end{proof}
\section{Proof of theorem \ref{thm:main} for number fields}\label{sec:sieving in groups}
In this section we prove the following theorem which is the special case of Theorem \ref{thm:main} under the assumption that $\dsF=k$ is a number fields. We recall that a subset $\Sigma$ in a group $\Gamma$ is called admissible if $\Sigma=\Sigma^{-1}$ and satisfying that $Cay(\Gamma,\Sigma)$ is not bipartite. This can be achieved if the elements of $\Sigma$ satisfy an odd relation. In general this is true if $\Sigma$ is not contained in the complement of a subgroup of index 2.
\begin{thm}\label{thm:main thm use sieve method}
Let $k$ be a number field, $n\in\dsN$, and $\Sigma\subset\Gln(k)$ a finite admissible subset. Let $\Gamma:=\langle\Sigma\rangle$, and $H:=\overline\Gamma$ with $H^o$ semisimple. For $\gamma\in\Gamma$ let $\Pi(\gamma H^o)$ be the group $\Pi(\gamma H^o,k)$ defined in \S\ref{sec:F-Cartans}. Let $w:\dsN\to\Sigma$ be a random walk of $\Gamma$. Then there exists $c>0$, such that for all $k\in\dsN$
$$
\dsP(w_k\in\left\{\gamma\in\Gamma:\Gal(k(\gamma)/k)\ne\Pi(\gamma H^o)\right\})\ll e^{-ck}
$$
\end{thm}
We first recall the basic setting from the previous sections. Let $H$ be an algebraic group with $H^o$ semisimple. Let $H_i$ be a coset of $H^o$. For a semisimple regular element $h\in H_i$, let $C_h$ be the unique Cartan subgroup associated to $H_i$ that contains $h$. We have constructed a homomorphism $\phi_h:\Gal(\ck/k)\to\Pi(H_i,C_h)$, with $\ker(\phi_h)$ satisfying $\Gal(k_{C_h}/k)\simeq \Gal(\ck/k)/\ker(\phi_h)$. Furthermore, if $H_i$ splits over $k$, then $\Pi(H_i,C_h,k)=W(H_i,C_h)$ the outer Weyl group associated to $H_i$, and in general $\Pi(H_i,C_h,k)=\langle\phi_h(\Gal(k_{C_h}/k)),W(H_i,C_h)\rangle$.
In order to justify our focus on Cartan subgroups instead of elements we need the following definition generalizing similar notions from \cite{PR}.
\begin{defn}\label{defn:H_i-k irred}
Let $H,H^o,H_i$ and $k$ be as above. Let $H^o=H^1\cdots H^l$ be the decomposition of $H^o$ as an almost produck of $k$-simple subgroups, and $H^o=H^{(1,i)}\cdots H^{(r,i)}$ be the minimal decomposition of $H^o$ into normal subgroups of $H^o$ invariant under conjugation by $H_i$ (see Appendix \ref{app:k-q Cartans}). We say that
\begin{enumerate}
\item
A $k$-Cartan subgroup $C$ associated to $H_i$ is $H_i-k$ \textit{quasi irreducible} if $C^o$ has no $k$-subtori other than $C^o\cap H^{(j,i)},j=1,\dots,r$.
\item
An element $x\in H^o$ is {\it{without $H_i$-components of finite order}} if for some (equivalently, any) decomposition $x=x_1\cdots x_r$ with $x_i\in H^{(s,i)},s=1,\dots,r$, all the $x_i$'s have infinite order.
\end{enumerate}
\end{defn}
In Appendix \ref{app:k-q Cartans} we prove (see Lemma \ref{lem:k-q_ired_Cartans} and Corollary \ref{cor:k-q_ired_elements}) the following generalization of a result in \cite{PR}:
\begin{lem}\label{lem:kq_ired_Cartans}
Let $H$ be a linear algebraic group defined over a number field $k$, such that $H^o$ is semisimple. Fix a coset $H_i$ of $H^o$, and let $H=H^{(1,i)}\cdots H^{(r,i)}$ as above. Let $C$ be a $k$-Cartan subgroup associated to $H_i$ such that $\Gal(k_C/k)$ contains the Weyl group $W(H_i,C)$, then $C^o$ is $H_i-k$ quasi irreducible. Furthermore, if $x\in C\cap H_i(k)$ is such that $x^{\ord(H_i)}$ is without $H_i$-components of finite order (where $\ord(H_i)$ is the order of $H_i$ in $H/H^o$), then $\Gal(k(x)/k)=\Gal(k_C/k)$.
\end{lem}

Theorem \ref{thm:main thm use sieve method} follows from the following two lemmas
\begin{lem}\label{lem:sieve4regss}
Let $\Sigma,\Gamma,H,H^o,H_i,w:\dsN\to\Sigma$ be as above, and assume $H^o$ is semisimple. Then there exists $c_1>0$ such that for all $k\in\dsN$
\begin{equation}\label{eq:sieve4regss}
\dsP(w_k\in\left\{\gamma\in\Gamma:\gamma\text{ is not regular semisimple }\right\})\ll e^{-c_1k}
\end{equation}
and
\begin{equation}\label{eq:cmpnnts_finite_order}
\dsP(w_k\in\left\{\gamma\in\Gamma\cap H_i:\gamma\text{ has an }H_i\text{-component of finite order}\right\})\ll e^{-c_1k}
\end{equation}
\end{lem}
\begin{lem}\label{lem:sieve4conjcls}
Let $\Sigma,\Gamma,H,H^o,H_i,w:\dsN\to\Sigma$ be as above, and assume $H^o$ is semisimple, and $H_i$ splits over $k$. For $U\in W(H_i)^\sharp$ denote
$$
Z_1=\left\{\gamma\in\Gamma:\gamma\text{ is semisimple regular, and }\phi_\gamma(\Gal(\ck/k))\cap U=\emptyset\right\}.
$$
 There exists $c_2>0$ such that for all $n\in\dsN$
\begin{equation}\label{eq:sieve4conjcls}
\dsP(w_k\in Z_1)\ll e^{-c_2n}
\end{equation}
\end{lem}

We first conclude the proof of the Theorem \ref{thm:main thm use sieve method} assuming the above Lemmas, by showing that the complement of the set
$$
\left\{\gamma\in\Gamma:\Gal(k(\gamma)/k)=\Pi(\gamma H^o)\right\}
$$
is contained in a finite union of exponentially small sets. By Lemma \ref{lem:sieve4regss}, we may assume that if $\gamma\in H_i$ then $\gamma$ is regular semisimple without $H_i$-components of finite order. In particular, for such $\gamma\in\Gamma\cap H_i$, $C_\gamma:=\langle\left(Z_{H^o}(\gamma)\right)^o,\gamma\rangle$ is the unique $k$-Cartan subgroup associated to $H_i$ containing $\gamma$, and by Lemma \ref{lem:kq_ired_Cartans}, that $\Gal(k_{C_\gamma}/k)=\Gal(k(\gamma)/k)$. We therefore concentrate our attention from now on to elements in Cartan subgroups, and Galois groups of splitting fields of them. We want to show that for any coset $H_i$, most elements $C$ of $\mathcal C_i(k)$ satisfy $\Gal(k_C/k)=\Pi(H_i,k)$. By Lemma \ref{lem:splitng fld}\eqref{prop:functorPi} $\Gal(k_C/k)=\Pi(H_i,k)$ if there exists a finite extension $K$ of $k$, such that $\Gal(K_C/K)=\Pi(H_i,K)$. We may therefore assume, by replacing the field $k$ with a finite extension, that all cosets of $H^o$ split over $k$. In particular, we have that $\Pi(H_i,k)=\W(H_i)$.
Using a well known theorem of Jordan, stating that for any finite group, a proper subgroup must have an empty intersection with at least one conjugacy class, we have that if $C\in\mathcal C_i(k)$ satisfies that the image of $\Gal(k_C/k)\ne \W(H_i,C)$, then there exists a conjugacy class $U\in \W(H_i)^\sharp$, such that for any $\sigma\in \Gal(\ck/k)$, $\theta_\sigma(C)\ne U$. We therefore have the following inclusion:
\begin{multline}\label{eq:galois as sieve1}
\left\{\gamma\in\Gamma:\Gal(k(\gamma)/k)\not\simeq \Pi(\gamma H^o)\right\}\subset\\
\bigcup_{i=1}^m \left\{\gamma\in\Gamma\cap H_i:\gamma\not\in\left(H_i(k)\right)_{rs}\text{ or has an }H_i\text{-component of finite order}\right\}\bigcup\\
\bigcup_{U\in W(H_i)^\sharp}\!\!\!\!\!\left\{\gamma\in\Gamma\cap\left(H_i(k)\right)_{rs}:\substack{\gamma\text{ is without }H_i\text{-components of finite order, }\\\phi_\gamma(\Gal(\ck/k))\cap U=\emptyset}\right\}
\end{multline}
and since this is a finite union, it suffice to show that each set is exponentially small, which is the content of Lemma \ref{lem:sieve4conjcls} and Lemma \ref{lem:sieve4regss}, and thus the Theorem is proved.
\subsection{Sieve method}
For both Lemmas the following theorem of Salehi Golsefidi-Varju \cite{SGV} is essential.
 \begin{thm}\label{thm:SGV}
 Let $\Gamma\subset\Gln(\mathds Z[\frac1{q_0}])$ be the group generated by a symmetric set $S$. Then $\mathcal G(\pi_q(\Gamma),\pi_q(S))$ form a family of expanders when $q$ ranges over square free integers coprime to $q_0$ if and only if the connected component of the Zariski closure of $\Gamma$ is perfect.
 \end{thm}
\subsubsection{Proof of Lemma \ref{lem:sieve4regss}}
Assume, without loss of generality, that $H=\langle H^o,H_i\rangle$, and let $\pi_{(j,i)}:H\to H/\widehat{H^{o}_{i,j}}$, where $$
\widehat{H^o_{i,j}}:=H^{(1,i)}\cdots H^{(j-1,i)}H^{(j+1,i)}\cdots H^{(r,i)},
 $$
(that is the product with $H^{(j,i)}$ dropped) be the reduction map from $H$ onto the adjoint representation of $H^{(j,i)}$. Notice that if $x=x_1\cdots x_r\in H^o$ is such that $x_j$ is of finite order, then $\pi_{(j,i)}(x)$ has finite order. The proof will follow from Proposition \ref{prop:reduction mod primes} and the following Proposition 2.7 from \cite{LM}:
\begin{prop}\label{prop:sieve subvar}
Let $\Gamma$ be a finitely generated subgroup of $\Gln(\dsQ)$, such that the identity component of the Zariski closure of $H=\overline\Gamma$ is perfect. Let $V$ be a proper subvariety of $H$, defined over $\dsQ$. Then the set $V(\dsC)\cap\Gamma$ is exponentially small.
\end{prop}
By Proposition \ref{prop:most lmnts rs}, the set of regular semisimple elements in $H_i$ contain an open dense subset, defined over $k$, and therefore \eqref{eq:sieve4regss} is a consequence of Proposition \ref{prop:sieve subvar}. Notice, also, that for a fixed field $k$, there exists $m\in\dsN$, such that if $h\in H(k)$ with eigenvalue that is a root of unity, then the order of the root divides $m$. Therefore all regular semisimple elements of finite order in $H_i$ are inside the variety defined by $x^{mn}=1$, where $n$ is the order of $H_i$ in $H/H^o$, and thus \eqref{eq:cmpnnts_finite_order} also follows from Proposition \ref{prop:sieve subvar}.
\subsubsection{Proof of Lemma \ref{lem:sieve4conjcls}}
Lemma \ref{lem:sieve4conjcls} is a consequence of the sieve method for groups, and in particular the following Theorem:
\begin{thm}[\cite{LM} Corollary 3.3]\label{thm:sieve method}
Fix $s\geq2$. Let $\Gamma$ be a finitely generated group, and let $\Sigma\subset\Gamma$ an admissible subset of it. Let $\Lambda<\Gamma$ be a subgroup of $\Gamma$ of finite index, and let $(N_i)_{i\geq s}$ be a sequence normal subgroups of $\Gamma$ of
finite index which are contained in $\Lambda$. Let $Z\subset\Gamma$ and assume the following
\begin{enumerate}
\item \label{thm:sieve method tau} $\Gamma$ has property-$\tau$ w.r.t the series of
normal subgroup $(N_i \cap N_j)_{i,j\geq s}$.
\item \label{thm:sieve method poly} There exists $d$ such that $|\Gamma/N_j| \le j^d$ for every $j\geq s$.
\item \label{thm:sieve method disj} $|\Lambda/\left(N_i\cap N_j\right)|=|\Lambda/N_i||\Lambda/N_j|$ for every $i\ne j\geq s$
\item \label{thm:sieve method dens} There is $c>0$ such that for every coset $\kappa\in\Gamma/\Lambda$ and every $j\geq s$
\begin{equation}\label{eq:sieve density condition}
|\left(Z\cap\kappa\right)N_j/N_j|\leq (1-c)|\Lambda/N_j|.
\end{equation}
\end{enumerate}
Then there exist $\alpha,t>0$ such that
$$
\dsP\left(\left\{w_k\in \Gamma:w_k\in Z\right\}\right)\leq e^{-\alpha k}\quad \forall k\geq tlog s
$$
\end{thm}
By Proposition \ref{prop:reduction mod primes} there exists an open non-empty subset $V$ defined over $k$ and a finite set of primes $S$, such that the following set
\begin{equation*}
Z=\left\{h\in\Gamma: h\in V(k)\subset \left(H_i(k)\right)_{rs},\;\theta_{\Frob_\ideal p}(h)\neq U,\forall\ideal p\not\in S\right\}.
\end{equation*}
is contained in the following
\begin{equation*}
 Z_1:=\left\{h\in\Gamma: \pi_{\ideal p}(h)\in V(\dsF_{\ideal p})\subset\left(\pi_{\ideal p}(\Gamma)\right)_{rs},\theta_{\Frob_\ideal p}(\pi_{\ideal p}(h))\neq U,\forall\ideal p\not\in S\right\}
\end{equation*}
and thus it suffice to show that the latter is exponentially small. For this goal we will use Theorem \ref{thm:sieve method}. In order to define the sequence of normal subgroups $\{N_i\}_{i\geq s}$, and the subgroup $\Lambda<\Gamma$ we use the following Proposition
\begin{prop}\label{prop:primes for sieve}
Let $k$ be a number field, $\Delta=\langle g_1,\dots,g_r\rangle,g_i\in \Gln(k)$ a finitely generated group, such that $G:=\overline\Delta$ is connected and semisimple. Let $\psi:\widetilde{G}\to G$ be the simply connected cover of $G$.
Then there exists a finite set $S$ of primes of $\mathcal O_k$,  the ring of integers of $k$, and a set $\mathcal P$ of primes of positive density such that for any $\ideal p\in\mathcal P$
$$\pi_{\ideal p}(\Delta)\subset\psi(\widetilde{G}(\dsF_{\ideal p}))$$
\end{prop}
\begin{proof}
Let $\widetilde\Delta:=\psi^{-1}(\Delta)$. This is a subgroup of $\psi^{-1}(G(k))\subset\widetilde{G}(K)$ where $K$ is a finite extension of $k$. Since $\widetilde\Delta$ is finitely generated the entries of its elements are $S$-arithmetic for a finite set of primes of $K$. By Chebotarev's density Theorem (c.f. p. 143 in \cite{IKbook}), there exists a set of positive density of primes $\mathcal P$ of $\mathcal O_k$, that totally splits in $K$, that is for any prime ideal $\ideal P\lhd\mathcal O_{K,S}$ such that $\ideal P|\ideal p$, $\mathcal O_{K,S}/\ideal P\simeq \mathcal O_k/\ideal p\simeq\dsF_{\ideal p}$ and thus when reducing $\Delta$ mod $\ideal p$ we have
$$\pi_{\ideal p}(\Delta)\subset\psi(\widetilde{G}(\dsF_{\ideal p}))$$
\end{proof}
We apply Proposition \ref{prop:primes for sieve} with $G=H^o$, $\Delta=\Gamma\cap H^o$.
For the set $\mathcal P$ given by Proposition \ref{prop:primes for sieve}, set $\left\{N_{\ideal p}\right\}_{\ideal p\in\mathcal P}$, and $\Lambda:=\psi(\psi^{-1}(\Delta)\cap\widetilde{H^o})$. It is therefore left to show that all conditions of Theorem \ref{thm:sieve method} hold:
\begin{enumerate}
    \item Condition \ref{thm:sieve method tau} holds by Theorem \ref{thm:SGV}.
    \item Condition \ref{thm:sieve method poly} holds since $\mathcal P$ has positive density inside primes.
    \item Condition \ref{thm:sieve method disj} holds since by the Strong Approximation Theorem (see e.g. \cite{weisfeiler} Theorem 9.1.1,\cite{nori} Theorem 5.1, and Proposition 5.2 in \cite{JKZ} and the references therein) $\Lambda/N_{\ideal p}\simeq\widetilde H(\dsF_{p})$ and by the Chinese Remainder Theorem, $\Lambda/(N_{\ideal p}\cap N_{\ideal q})=\Lambda/N_{\ideal p}\times\Lambda/N_{\ideal q}\simeq\widetilde H(\dsF_{\ideal p})\times\widetilde H(\dsF_{\ideal q})$
    \item For Condition \ref{thm:sieve method dens}, let $\kappa\in\Gamma/\Lambda$, and let $H_i$ be the coset of $H^o$ such that $\kappa\in H_i$. Let $\ideal p\in\mathcal P$, and write $\kappa=\Lambda\gamma$, where $\gamma\in\kappa$. For a regular semisimple element $h\in\kappa$ such that $\pi_{\ideal p}(h)$ is also regular semisimple, denote $C_{h,\ideal p}$ the unique Cartan subgroup containing $\pi_{\ideal p}(h)$. Then
    \begin{multline}
|\pi_{\ideal p}(Z_1\cap\kappa)|\leq|\left\{h\in\pi_{\ideal p}(\Gamma)_{rs}: \theta_{\Frob_{\ideal p}}(h)\neq U\right\}|\leq\\
\left|\bigcup_{
\substack{
h:\pi_{\ideal p}(h)\in\left(H_i(\dsF_{\ideal p})\right)_{rs}\\
\theta_{Frob_{\ideal p}}(\pi_{\ideal p}(h))\neq U}
}\left(C_{h,\ideal p}\right)^o\pi_{\ideal p}(h)\cap\pi_{\ideal p}(\Lambda\gamma)\right|=\\
\left|\bigcup_{
\substack{h:\pi_{\ideal p}(h)\in \left(H_i(\dsF_{\ideal p})\right)_{rs}\\\theta_{Frob_{\ideal p}}(\pi_{\ideal p}(h))\neq U}
}\left(C_{h,\ideal p}\right)^o\cap\pi_{\ideal p}(\Lambda)\right|
\end{multline}
Where the last equality is since $\pi_{\ideal p}(\gamma h^{-1})\in\pi_{\ideal p}(\Lambda)$ by Proposition \ref{prop:primes for sieve}. Let $C_1\in\mathcal C_i(\dsF)$ be a fixed Cartan subgroup (provided by Lemma \ref{lem:fin fld conj corsp}) satisfying $\theta(C_1)=U$. Since $H_i(\dsF_{\ideal p})$ splits, and by Lemma \ref{lem:fin fld conj corsp} the complement set of the latter set contains the following union of sets
\begin{multline*}
\left|\bigcup_{C\in\mathcal C_i(\dsF_{\ideal p}):\theta_{\Frob_{\ideal p}}(C)\in U}\left(C^o\cap\widetilde{H}(\dsF_{\ideal p})\right)\right|\geq \left|\bigcup_{g\in H^o(\dsF_{\ideal p})/N_{H^o}(C_1)(\dsF_{\ideal p})} C_1^o\cap\widetilde H(\dsF_{\ideal p})\right|\geq\\
\frac{|C_1^o\cap\psi\left(\widetilde{H}(\dsF_{\ideal p})\right)||H^o(\dsF_{\ideal p})|}{|N_{H^o}(C_1)(\dsF_{\ideal p})|}\geq\frac{|C_1^o| |\widetilde{H}(\dsF_{\ideal p})||H(\dsF_{\ideal p})|}{|N_{H^o}(C_1)(\dsF_{\ideal p})|}.
\end{multline*}
In the last inequality we used the fact that $\psi\left(\widetilde H(\dsF_{\ideal p})\right)$ is a normal subgroup of $H^o(\dsF_{\ideal p})$, and that
$$
|C_1^o\cap\psi\left(\widetilde H(\dsF_{\ideal p})\right)|\geq\frac{|C_1^o(\dsF_{\ideal p})||\psi\left(\widetilde H(\dsF_{\ideal p})\right)}{|H(\dsF_{\ideal p})|}.
$$
By Lemma \ref{lem:Weyl centralizers} and as $\widetilde H(\dsF_{\ideal p})$ is a normal subgroup of bounded index in $H^o(\dsF_{\ideal p})$ we get that there exists a constant $c>0$ such that the last expression is bounded below by $c|\psi\left(\widetilde H(\dsF_{\ideal p})\right)|$. We thus find that $1-|\pi_{\ideal p}(Z)|/|\pi_{\ideal p}(\Lambda)|\geq c>0$, which shows that condition \ref{thm:sieve method dens} holds.
\end{enumerate}
Since all conditions of Theorem \ref{thm:sieve method} hold, this concludes the proof of Lemma \ref{lem:sieve4conjcls}.
\subsection{Arbitrary groups}
We conclude this section by proving Theorem \ref{thm:main} over number fields, for any group $H$ such that $H^o$ does not contain a central torus. The following lemma was proved in \cite{JKZ} (Lemma 2.3)
\begin{lem}\label{lem:reduction to ss}
Let $k$ be a field of characteristic 0, and $\Gamma<\Gln(k)$ a subgroup. Let $G=\overline{\Gamma}$ be the Zariski closure of $\Gamma$, $R_u(G)$ the unipotent radical of $G^o$, and $\pi_u:G\to G/R_u(G)$ the reduction map. Then for any $g\in G(k)$, $k(g)=k(\pi_u(g))$.
\end{lem}
Let $\Gamma'=\pi_u(\Gamma)=\langle\pi_u(\Sigma)\rangle$ where $\pi_u:H\to H/R_u(H^o)$. Then by the assumption that $H^o$ does not contain a central torus we get that $\pi_u(H^o)$ is semisimple. By Lemma \ref{lem:reduction to ss}, $k(\gamma)=k(\pi_u(\gamma))$. For a coset $H_i$ of $H^o$, let $\Pi(H_i,k)=\Pi(\pi_u(H_i),k)$. Then we get that
\begin{multline*}
\dsP(w_k\in\left\{\gamma\in\Gamma:\Gal(k(\gamma)/k)\ne\Pi(\gamma H^o)\right\})\leq\\
\dsP(w_k\in\left\{\pi_u(\gamma)\in\Gamma':\Gal(k(\pi_u(\gamma))/k)\ne\Pi(\pi_u(\gamma H^o))\right\})
\end{multline*}
which by Theorem \ref{thm:main thm use sieve method} is exponentially small.
\section{From number fields to finitely generated fields}\label{sec:general case}
In the previous section we proved Theorem \ref{thm:main} for the
case when $\dsF=k$ is a number field. The goal of this section is to
prove the same result for the general case of finitely generated
field $\dsF$. This will be done using the procedure of
specialization.

Let us recall the notations of \S \ref{sec:Cartans}: Let $\dsF$ be
any characteristic zero field, $H$ an algebraic subgroup of $GL_n$
defined over $\dsF$, $H^o$ its connected component and
$\{H_i\}_{i=1}^m$ the cosets of $H^o$ in $H$. For every $1\leq i\leq
m$ we defined a group $\Pi(H_i,\dsF)$ and we showed (Proposition
\ref{prop:Gal subqtint Pi}) that for every $h\in H_i(\dsF)$,
$\Gal(\dsF(h)/\dsF)$ is isomorphic to a quotient of a subgroup of
$\Pi(H_i,\dsF)$. By Lemma \ref{lem:reduction to ss}, we can assume
that the connected component of the Zariski closure of $\Gamma$ is
semisimple. In \S\ref{sec:sieving in groups} we proved Theorem
\ref{thm:main} for the case where $\dsF=k$ is a number field and
$H^o$ is a semisimple group, namely we showed that if
$\Gamma=\langle\Sigma\rangle$ is a finitely generated subgroup of
$H(k)$, then outside an exponentially small subset of $\Gamma$,
$\Gal(k(\gamma)/k)$ is isomorphic to $\Pi(H^o\gamma,k)$.

Now, if we turn to the general case of finitely generated field
$\dsF$, and $\Gamma=\langle\Sigma\rangle$ is a Zariski dense
subgroup of $H(\dsF)\subset GL_n(\dsF)$, then $\Gamma\subset
GL_n(A)$ where $A$ is some finitely generated $\dsQ$-algebra of
$\dsF$, with $\dsF$ being its field of quotients. If $\varphi:A\to
k$ is a ring homomorphism (in fact $\dsQ$-algebra homomorphism) onto
a number field (such $\varphi$ is called a \textit{specialization}),
then $\varphi$ induces a group homomorphism $\varphi:GL_n(A)\to
GL_n(k)$,and for every $h\in GL_n(A)$, $\Gal(k(\varphi(h))/k)$ is a
quotient of $\Gal(\dsF(h)/\dsF)$. Let $\widetilde\Gamma$ be the
image of $\Gamma$, so it is generated by
$\widetilde\Sigma=\varphi(\Sigma)$, and $\widetilde{H}$ be the
Zariski closure of $\widetilde\Gamma$. In Proposition \ref{prop:spec
construction} below, we will show that for any coset $H_i$ there
exists such a specialization $\varphi_i$ with $\widetilde{\langle
H^o,H_i\rangle}$ isomorphic to $\langle H^o,H_i\rangle$, and such
that $\Pi(\widetilde{H_i},k)=\Pi(H_i,\dsF)$. Once Proposition
\ref{prop:spec construction} is proved, Theorem \ref{thm:main} holds
for any finitely generated $\dsF$ and $H^o$ semisimple. Indeed, we
know on one hand that for $h\in H_i$, $\Gal(\dsF(h)/\dsF)$ is a
subquotient of $\Pi(H_i,\dsF)$. On the other hand
$\Gal(k(\varphi_i(h))/k)$ which is a quotient of
$\Gal(\dsF(h)/\dsF)$ is isomorphic , for \textit{almost all} $h$, to
$\Pi(\widetilde{H_i},k)\simeq\Pi(H_i,\dsF)$. This implies that
$\Gal(\dsF(h)/\dsF)$ is isomorphic to $\Pi(H_i,\dsF)$ for
\textit{almost all} $h\in H$, i.e., outside an exponentially small
subset.

We are left to prove:
\begin{prop}\label{prop:spec construction}
Let $\dsF=\dsQ(y_1,\dots,y_r)$ be a finitely generated field,
$\Gamma\leq GL_n(\dsF)$ a finitely generated with semisimple Zariski
closure $H$. For a connected component $H_i$ denote
$\Gamma_i:=\Gamma\cap\langle H^o,H_i\rangle$. Then for any connected
component $H_i$, there exist a finitely generated $\dsQ$-algebra
$A_i=\dsQ[x_1,\dots,x_r]$ and a $\dsQ$-algebra homomorphism
$\varphi_i$ from $A_i$ to a number field $k_i$, inducing a group
homomorphism $\tilde\varphi_i:GL_n(A_i)\to GL_n(k_i)$ such that:
\begin{enumerate}
\item\label{prop:spec construction1}
$\Gamma_i\leq GL_n(A_i)$
\item\label{prop:spec construction2}
$\dsF$ is the field of quotients of $A_i$.
\item\label{prop:spec construction3}
If $\widetilde H$ denotes the Zariski closure of $\tilde\varphi_i(\Gamma_i)$, then $\tilde\varphi_i$ induces an isomorphism between $\Gamma_i/H^o\cap\Gamma$ and $\tilde\varphi_i(\Gamma_i)/\widetilde{H^o}\cap\tilde\varphi_i(\Gamma_i)$ and (as $\Gamma_i$ is Zariski dense in $\langle H^o,H_i\rangle$ and $\tilde\varphi_i(\Gamma_i)$ in $\widetilde{H}$) also between $\langle H^o,H_i\rangle/H^o$ and $\widetilde H/\widetilde{H^o}$ which we also denote by $\tilde\varphi_i$.
\item\label{prop:spec construction4}
$\Pi(H_i,\dsF)\simeq\Pi(\widetilde{H_i},k)$
\end{enumerate}
\end{prop}
The proof of Proposition \ref{prop:spec construction} is a consequence of the following lemmas
\begin{lem}\label{lem:hilberianity}
Let $K$ be a number field, $R=K[x_1,\dots,x_l]$ a finitely generated
integral domain over $K$, $F=Quot(R)$ the field of quotients of $R$,
$E$ a finite Galois extension of $F$ and $S$ the integral closure of
$R$ in $E$. Then there exists a $K$-epimorphism $\varphi$ from $S$
onto a finite extension $\overline E$ of $K$, such that if we denote
$\overline F=\varphi(R)$, then $\overline E$ is a Galois extension
of $\overline F$ and $\Gal(\overline E/\overline F)=\Gal(E/F)$.
\end{lem}
\begin{proof}
By Noether's normalization Lemma, let $t_1,\dots,t_r\in R$ be
algebraically independent over $K$, such that $R$ is integral over
$R_0=K[t_1,\dots,t_r]$ (cf. \cite{FJ} Proposition 5.2.1). Denote
$F_0=Quot(R_0)$. Let $z\in S$ be a primitive element for $E/F_0$,
and let $f=f(t_1,\dots,t_r,Z)\in R_0[Z]$ be the irreducible
polynomial of $z$ over $F_0$. Since $K$ is a Hilbertian field, there
exists $a=(a_1,\dots,a_r)\in K^r$ such that $f(a,Z)$ is separable
and irreducible over $K$, and $\deg(f(a,Z))=\deg(f(t,Z))$. Since $S$
is integral over $R_0$, we can extend the specialization $t\mapsto
a$ to a $K$-homomorphism $\varphi$ from $S$ onto a finite extension
$\overline E$ of $K$. Denote by $\overline F=\varphi(R)$, and get
that
\begin{equation}\label{eq:hilbertianity}
[E:F_0]\geq [\overline E:K]\geq \deg(f(a,Z))=\deg(f(t,Z))=[E:F_0]
\end{equation}
and therefore $[E:F_0]=[\overline E:K]$. Notice that $[\overline
F:K]\leq[F:F_0]$ and $[\overline E:\overline F]\leq[E:F]$. Now,
since $[\overline E:K]=[\overline E:\overline F][\overline F:K]$,
and $[E:F_0]=[E:F][F:F_0]$ we get from \eqref{eq:hilbertianity} that
$[\overline E:\overline F]=[E:F]$, and by Lemma 6.1 in \cite{FJ}
$[\overline E:\overline F]$ is a Galois extension, and $\varphi$
induces an isomorphism $\Gal(E/F)\simeq \Gal(\overline E/\overline
F)$.
\end{proof}
\begin{lem}\label{lem:gp hilbert}
Let $A$ be a finitely generated integral domain over $\dsQ$,
$\Gamma\leq GL_n(A)$ a finitely generated group with Zariski closure
$H$ satisfying $H^o$ is semisimple and $H/H^o$ is cyclic. Then there
is a number field $k$ and a $\dsQ$-homomorphism $\varphi:A\to k$,
such that the Zariski closure of $\varphi(\Gamma)$ is isomorphic to
$H$.
\end{lem}
\begin{proof}
This is proved in \cite{LL} Theorem 4.1 for the case where $H$ is a
connected simple group. The proof works word by word if $H$ is
semisimple and connected , as long as we are in the characteristic
zero. (The remark there before Theorem 4.1 warns that the proof is
more complicated for semisimple groups, but this is because of the
positive characteristic cases. In these cases $H$ has many finite
subgroups, which in characteristic zero all of them lie in a proper
subvariety, while in positive characteristic they do not). In fact
the proof in \cite{LL} shows that the result holds for an open
subset of $Spec(A)$.

Now, to get the non-connected case, we use Proposition \ref{prop:non
connected cyclic structure}, to get that if $H/H^o$ is cyclic, then
there exists a finite cyclic group $J$ such that $H=H^o\rtimes J$.
We can assume $J$ is inside $GL_n(K)$ where $K$ is some number
field, and we can replace $A$ by $KA$, and use \cite{LL} Theorem 4.1
when this time we take a $K$-homomorphism from $KA$ to $k$, a number
field containing $K$. It is easy to see that different connected
components give different cosets of $\varphi(\Gamma)^o$, as these
cosets are represented by elements of $J$ on which $\varphi$ acts
trivially.
\end{proof}
We are now ready to prove Proposition \ref{prop:spec construction}:
Fix a connected component $H_i$, and assume $H=\langle
H^o,H_i\rangle$. We can clearly choose $A_i\subset\dsF$ satisfying
\ref{prop:spec construction1} and \ref{prop:spec construction2}
since $\Gamma$ and $\dsF$ are finitely generated. Lemma \ref{lem:gp
hilbert} shows that we can arrange for $\varphi_i$ for which
$\tilde\varphi_i$ satisfies also \ref{prop:spec construction3}.
Moreover, this is so for an open subset of $Spec(KA)$ where $K$ is
the number field defined in the proof. Finally we recall that for
every connected component $H_i$, $\Pi(H_i,\dsF)$ is satisfies the
following exact sequence:
$$
1\to\W(H_i,C)\to\Pi(H_i,C,\dsF)\to\Gal(\dsF_i^W/\dsF)\to1,
$$
where $\dsF_i^W$ is the fixed field of the inverse image of
$\W(H_i,C)$, under the map
$\phi_C:\Gal(\overline{\dsF}/\dsF)\to\Pi(H_i,C,\dsF)$. Since
$H=H^o\rtimes J$ for a finite cyclic group defined over $K$, we get
that the action of every connected component on $H^o$ is
$K$-rational, and therefore it is the same on
$\widetilde{H^o},\W(\widetilde H^o)$, and in particular
$\W(H_i)=\W(\widetilde{H_i})$. It is therefore left to show that
there exists a specialization $\varphi$ such that
$\Gal(\dsF_i/\dsF)\simeq \Gal(k_i^W/k)$, where $k_i$  is the fixed
field of the inverse image of $\W(\widetilde
H_i,\widetilde{\varphi}(C))$. We first notice that since $\varphi$
is onto, any Cartan subgroup $\widetilde C<\widetilde H$ is the
image of a Cartan subgroup $C<H$. Also, if $\dsF_C$ is the splitting
of $C$ which is the splitting field of a polynomial
$f_C(T)\in\dsF[T]$, then the splitting field of $\varphi(f_C)(T)\in
k[T]$ is the splitting field of $\widetilde C$, and if
$\sigma\in\Gal(k_C/k)$ satisfies that
$\phi_{\widetilde{C}}(\sigma)\in\W(H_i,\widetilde C)$, then
$\phi_C(\sigma)\in\W(H_i,C)$. In particular, if we denote by
$f_i(T)\in\dsF[T]$ to be a polynomial such that its splitting field
is $\dsF_i^W$, then the splitting field of $\varphi(f_i)(T)\in k[T]$
is $k_i^W$. Now, let $S_i$ be the integral closure of $A_i$ in
$F_i^W$. By Lemma \ref{lem:hilberianity}, there is a specialization
of $S_i$ onto a number field $k$, such that
$\Gal(\dsF_i^W/\dsF)=\Gal(\varphi(S_i)/\varphi(A))$. Moreover, the
set of such specializations is an Hilbertian set (cf. \cite{FJ}, ch.
12) and so Zariski dense, hence has a non trivial intersection with
the set of specializations given by Lemma \ref{lem:gp hilbert} which
ensures that the Zariski closure of $\tilde\varphi(\Gamma)$ is
isomorphic to $H$. Now for any fixed connected component $H_i$,
$$
\xymatrix{
 1\ar[r]&\W(H_i,C)\ar[r]&\Pi(H_i,C,\dsF)\ar[r]&\Gal(\dsF_i^W/\dsF)\ar[r]&1\\
 1\ar[r]&\W(H_i,\widetilde C)\ar[r]\ar[u]&\Pi(H_i,\widetilde
 C,k)\ar[r],\ar[u]&\Gal(k_i^W/k)\ar[r]\ar[u]&1
 }
$$
where all arrows from the bottom sequence above are induced by
$\varphi$. By construction we get that all arrows are isomorphisms,
and thus $\Pi(H_i,C,\dsF),\Pi(\widetilde H_i,\widetilde C,k)$ are
isomorphic, and Proposition \ref{prop:spec construction} is proven.
\section{Examples and counter examples}\label{sec:examples}
In order to give a better understanding for the types of Galois groups constructed by Theorem \ref{thm:main}, we give, in this section, some examples. We end this section with some counter examples to Theorem \ref{thm:main} for different types of groups and fields for which one of the conditions of Theorem \ref{thm:main} fails.
\subsection{Connected groups with $\Pi$ larger than the Weyl group} As a first example, we show how it can happen that the group $H$ is connected, but $\Pi(H,\dsF)\neq W(H)$. Let $\dsF=\dsQ$, $k$ a number field, and $L$ its Galois closure. Let $\Gamma=Res^{k}_{\dsQ}(SL_n)(\dsZ)$. Note that $\Gamma$ is isomorphic to $SL_n(\mathcal O_k)$, where $\mathcal O_k$ is the ring of integers of $k$, but we take it as a subgroup of $SL_{nd}(\dsQ)$. It is finitely generated, and its Zariski closure is $H$ which over $\overline{\dsQ}$ is isomorphic to $SL_n^d$, where $d$ is the degree of $k$ over $\dsQ$. The splitting field of $H$ is $L$, since $\Gal(\overline{\dsQ}/\dsQ)$ acts transitively on the absolutely simple components of $H$, which are the $SL_n$ components, and clearly $H$ splits over $L$. We therefore get that $\Pi(H,\dsQ)=\langle S_n^d,\Gal(L/\dsQ)\rangle\simeq S_nwr \Gal(L/\dsQ)$. This example can be generalized by replacing $SL_n$ with any other semisimple group $G$. In that case, $H=Res^k_{\dsQ}(G),\Gamma=H(\dsZ)$, and we get that $\Pi(H,\dsQ)\simeq W(G)wr \Gal(L/\dsQ)$.
\subsection{Non connected groups}
\subsubsection{Groups with irreducible Dynkin diagram}
Let $H:=SL_n\rtimes\langle\tau\rangle$, where $\tau$ is the automorphism of $SL_n$ sending $A\mapsto (A^t)^{-1}$, and $\Gamma:=H(\dsZ)$. A representation $\rho:H\to GL_{2n}$ is given by
$$
\rho(A)=\begin{pmatrix}A&0\\0&(A^t)^{-1}\end{pmatrix},\quad\rho(\tau)=\begin{pmatrix}0&I_n\\I_n&0\end{pmatrix}
$$
In order to have a better understanding of $W(H^o\tau)$, we first give examples of two Cartan subgroups associated to $H^o\tau$. Since $\tau$ is of finite order, then it is contained in a Cartan subgroup. Let $C_\tau$ be such a group. Then $C_\tau^o$ is a maximal torus of $Z_{SL_n}(\tau)=SO(n)$. Notice that if $n=2k+1$, then $W^\tau=W(SO(n))\simeq (\dsZ/2\dsZ)^k\rtimes S_k$, however if $n=2k$, then $W^\tau$ is still isomorphic to $(\dsZ/2\dsZ)^k\rtimes S_k$, but $W(SO(n))=(\dsZ/2\dsZ)_0^k\rtimes S_k$, where the subscript stands for vectors with sum zero. For the second example, assume $n=2k$ is even, and denote
$$
J=\begin{pmatrix} 0&I\\-I&0\end{pmatrix}
$$
Again $J\tau$ is of finite order, and therefore contained in a Cartan subgroup. Denote $C_{J\tau}$ such a group. Again $C_{J\tau}^o$ is a maximal torus in $Z_{SL_n}(J\tau)=Sp_n$. Notice that the centralizers of $\tau,J\tau$ are of different types, however they are of the same rank, $k$, and $W^{J\tau}=W(Sp_n)=(\dsZ/2\dsZ)^k\rtimes S_k$.
By Lemma \ref{lem:Weyl gp strctr}, for the connected component $H^o\tau$, the Weyl group is $W(H^o\tau)=\left(T_C/(T^\tau)\right)^\tau\rtimes W^\tau$ (notice that since $H^o$ is simply connected $W(H^o\tau)=\W(H^o\tau)$). As mentioned, $W^\tau$ in this case is $W_k=(\dsZ/2\dsZ)^k\rtimes S_k\simeq W(B_k)\simeq W(C_k)$, where $n=2k$ if $n$ is even or $n=2k+1$ if $n$ is odd, and $(T_C/(T^\tau))^\tau=(\dsZ/2\dsZ)^{r}$, where $r=n-1-k$ is the difference between the ranks of the $A_{n-1}$ type group and $B_k$ type. We therefore get that for both cases ,$n$ even or odd, $W(H^o\tau)=(\dsZ/2\dsZ)wr_{\Omega}W_k$, where $\Omega$ is a set of $k$ signed pairs, and $W_k$ acts on $\Omega$ in the standard way.
\subsubsection{Groups with reducible Dynkin diagram}
Let $H:=SL_n^d\rtimes\langle\tau\rangle$, where $\tau$ is the automorphism that cyclicly permutes the $SL_n$ factors, and $\Gamma:=H(\dsZ)$. A representation of $\rho:H\to GL_{nd}$ is given in a similar way to the previous example. Notice that this representation is in fact into $GL(\underbrace{V\oplus\cdots\oplus V}_{d\text{ times}})$ permuting the $V$ factors, where $V$ is $n$-dimensional. In particular, since $\rho(\tau)$ permutes the factors isomorphic to $V$ cyclicly, we find that if $\lambda$ is an eigenvalue of an element $\rho(A\tau)$, then $\lambda\zeta_d$ is also an eigenvalue for every $\zeta_d$ an $d$-th root of unity. In particular $\dsQ(\zeta_d)$ is contained in the splitting field of $H^o\tau$. The Weyl group in this case by a similar computation is $W(H^o\tau)=(\dsZ/d\dsZ)^{n-1}\rtimes S_n$, and $\Pi(H^o\tau,\dsQ)=\langle W(H^o\tau),\Gal(\dsQ(\zeta_d)/\dsQ)\rangle$.

The situation can be described by the following diagram of tower of fields extensions:
$$
\xymatrix{
\dsQ(\gamma)\ar@{-}_{\Gal(\dsQ(\gamma)/\dsQ(\gamma^d))\simeq\left(\dsZ/d\dsZ\right)^{n-1}}[d]\\
\dsQ(\gamma^d)\ar@{-}_{\Gal(\dsQ(\gamma^d)/\dsQ(\zeta_d))\simeq S_n}[d]\\
\dsQ(\zeta_d)\ar@{-}_{\Gal(\dsQ(\zeta_d)/\dsQ)\simeq \left(\dsZ/d\dsZ\right)^\times}[d]\\\dsQ}
$$
\subsection{Counter examples}
\subsubsection{Non semisimple groups}
Let $\Gamma:=\langle A^{\pm1},J\rangle$, where
$$
A=\begin{pmatrix}2& \\& 3 \end{pmatrix},\quad J=\begin{pmatrix}&1\\1&\end{pmatrix}
$$
Then $H=\overline\Gamma=\begin{pmatrix}*&\\&*\end{pmatrix}\bigcup\begin{pmatrix}&*\\ * &\end{pmatrix}$. Let $\Sigma:=\{A,A^{-1},J,I_2\}$ be an admissible generating set, where $I_2$ is the identity element, and let $X_k$ be the $k$-th step of the random walk generated by $\Sigma$. Notice that if $X_k=\begin{pmatrix}&a\\ b&\end{pmatrix}\not\in H^o$, then $\dsQ(X_k)=\dsQ(\sqrt{ab})$. Also, if $X_k\not\in H^o$, then if $X_k=A^{n_1}JA^{n_2}J\cdots=s_1s_2\cdots s_k\quad,s_i\in\Sigma $, then $J$ must appear an odd number of times, and therefore, if $k$ is even, $ab$ is a square if and only if $N_I:=|\{i:s_i=I\}|$ is odd, which occurs with probability $1/2-(1/2)^n$. On the other hand, if $k$ is odd, then $ab$ is a square if and only if $N_I$ is even which occurs with probability $1/2-(1/2)^n$. In particular we get that
\begin{eqnarray*}
\lim_{k\to\infty}\dsP(\Gal(\dsQ(X_{k})/\dsQ)=\Pi_1(H^oX_{k}))=\frac12\\
\lim_{k\to\infty}\dsP(\Gal(\dsQ(X_{k})/\dsQ)=\Pi_2(H^oX_{k}))=\frac12
\end{eqnarray*}
where
$$
\Pi_1(H^o\gamma)=\begin{cases}\{1\} & H^o\gamma=H^o\\\dsZ/2\dsZ & H^o\gamma\neq H^o\end{cases},\quad
\Pi_2(H^o\gamma)=\{e\}
$$
That is, there is no typical behaviour for the non connected component.
\subsubsection{Not finitely generated fields}
In this example, we show that if the field $\dsF$ is not finitely generated, then again it is possible that there is no typical behaviour for the Galois groups. We prove the following:
\begin{thm}\label{thm:not finitely gen counter ex}
Let $H:=SL_n$, $n\geq5$, $\Gamma:=H(\dsZ)$. Let $\Sigma$ be a finite generating set of $\Gamma$, and let $X_k$ be the corresponding random walk. Then for any pair of subgroups $\G=(G_1,G_2)$ of the alternating group $Alt(n)$ there exists an algebraic extension $\dsF_{\G}$ of $\dsQ$, and sequences $\{n_i(\G)\},\{k_i(\G)\}$, such that
\begin{eqnarray*}
\dsP(\Gal(\dsF_{\G}(X_{n_i(\G)})/\dsF_{\G})=G_1)\geq1-\frac1{2^{i}}\\
\dsP(\Gal(\dsF_{\G}(X_{k_i(\G)})/\dsF_{\G})=G_2)\geq1-\frac1{2^{i}}
\end{eqnarray*}
\end{thm}
In particular, this theorem shows that in this case, when the base field is not finitely generated, there is no generic behaviour.
\begin{proof}
Let $R_i=\{\sigma_{1,1},\dots,\sigma_{r_i,i}\},i=1,2$ be a generating set of $G_i$. Without loss of generality we may assume $r_1=r_2$, and therefore denote $|R_i|=r$. We will construct the field $\dsF_{\G}$ as a fixed field of Galois automorphisms $\sigma^{\G}_1,\dots,\sigma^{\G}_r\in \Gal(\overline{\dsQ}/\dsQ)$. Let $Q_{(2)}$ be the maximal pro-2 extension of $\dsQ$. Then for any $\gamma\in\Gamma$ such that $\Gal(\dsQ(\gamma)/\dsQ)=S_n$, we have $\Gal(\dsQ_{(2)}(\gamma)/\dsQ_{(2)})=A_n$. Since $A_n$ is simple, we get that $\dsQ_{(2)}(\gamma_1),\dsQ_{(2)}(\gamma_2)$ are linearly disjoint if and only if they are different.

Let $n_1\in\dsN$ be chosen, by applying Theorem \ref{thm:main} such that
$$
\dsP(\omega_{k_1}\in\{\gamma\in\Gamma:\Gal(\dsQ(\gamma)/\dsQ)=S_n\})>1-\frac1{2}
$$
Let $F_{1,1},\dots,F_{1,N_1}$ be the finite set of fields occurring as splitting fields over $\dsQ$ of all possible $n_1-th$ steps of the random walk with Galois group over isomorphic to $S_n$, and set $F_1$ to be their composite.

Let $k_1\in\dsN$ be chosen by applying Theorem \ref{thm:main} such that
$$
\dsP(\omega_{k_1}\in\{\gamma\in\Gamma:\Gal(F_1(\gamma)/F_1)=S_n\})>1-\frac1{2}
$$
Denote $E_{1,1},\dots,E_{1,M_1}$ the set of fields over $\dsQ$ occurring as splitting fields of all possible $k_1-th$ step of the random walk (notice that the random walk itself did not change, only the base field) with Galois group over $F_1$ isomorphic to $S_n$, and set $E_1$ to be the composite of $F_1$ with these fields. Notice that these fields are linearly disjoint from $F_{1,1},\dots,F_{1,N_1}$ over $\dsQ$, and $\dsQ_{(2)}E_{1,1},\dots,\dsQ_{(2)}E_{1,M_1}$ are linearly disjoint from $\dsQ_{(2)}F_{1,1},\dots,\dsQ_{(2)}F_{1,N_1}$ over $\dsQ_{(2)}$. 

We continue by induction: Let $n_{i+1}$ be such that
$$
\dsP(\omega_{k_1}\in\{\gamma\in\Gamma:\Gal(E_{i}(\gamma)/E_{i})=S_n\})>1-\frac1{2^{i+1}}
$$
Let $F_{i+1,1},\dots,F_{i+1,N_{i+1}}$ be the set of fields which are splitting fields over $\dsQ$ of all $n_{i+1}-th$ step of the random walk with Galois group isomorphic to $S_n$. Let $F_{i+1}$ be the composite of these fields with $E_i$. Continue as before to define $k_{i+1}$, $E_{i+1,1},\dots,E_{i+1,M_{i+1}}$, and $E_{i+1}$. Notice that at any step we add fields $F_{i,j}$ or $E_{i,j}$ that are linearly disjoint fields from the previous fields over $\dsQ$, since their Galois groups over $F_i$ or $E_i$ are isomorphic to $S_n$. Furthermore, when taking the composite of all fields with $\dsQ_{(2)}$, we get fields that are linearly disjoint over $\dsQ_{(2)}$. So all fields $\dsQ_{(2)}F_{i,j}$ and $\dsQ_{(2)}E_{i,j}$ are linearly disjoint over $\dsQ_{(2)}$. Also, by the construction of these fields, we have that
\begin{equation}\label{eq:infnit_count_ex1}
\dsP(\omega_{n_i}\in\{\gamma\in\Gamma:\exists 1\leq j\leq N_i,\dsQ(\gamma)=F_{i,j}\})>1-\frac1{2^i}
\end{equation}
\begin{equation}
\dsP(\omega_{k_i}\in\{\gamma\in\Gamma:\exists 1\leq j\leq M_i,\dsQ(\gamma)=E_{i,j}\})>1-\frac1{2^i}
\label{eq:infnit_count_ex2}
\end{equation}

Define $\sigma^{\G}_1,\dots,\sigma^{\G}_r\in \Gal(\overline{\dsQ}/\dsQ_{(2)})$ in the following way: By construction we know that $\Gal(\dsQ_{(2)}F_{i,j}/\dsQ_{(2)})\simeq \Gal(\dsQ_{(2)}E_{i,j}/\dsQ_{(2)})\simeq Alt(n)$. We may therefore fix for any such field an isomorphism to $Alt(n)$, and consider $\sigma_{d,i}\in \Gal(\dsQ_{(2)}F_{i,j}/\dsQ_{(2)})$ (or $\Gal(\dsQ_{(2)}E_{i,j}/\dsQ_{(2)})$) for $d=1,\dots,r,\;i=1,2$. Since furthermore the fields $\dsQ_{(2)}F_{i,j},\dsQ_{(2)}E_{i,j}$ are linearly disjoint over $\dsQ_{(2)}$ we can find $\sigma^{\G}_d\in \Gal(\overline{\dsQ}/\dsQ_{(2)})$ such that $\sigma^{\G}_d|_{F_{i,j}}=\sigma_{d,1},\forall i\in\dsN,j=1,\dots,N_i$, and $\sigma^{\G}_d|_{E_{i,j}}=\sigma_{d,2},\forall i\in\dsN,j=1,\dots,M_i$. Set $F=\{x\in\overline{\dsQ}:\sigma^{\G}_d(x)=x\; \forall d=1,\dots,r\}$. 

We claim now that for all $i\in\dsN,j=1,\dots,N_i$, $\Gal(FF_{i,j}/F)\simeq G_1$, and similarly $\Gal(FE_{i,j}/F)\simeq G_2$. To see this notice that by the fundamental theorem of Galois theory, $\Gal(\overline{\dsQ}/F)=\langle\sigma^{\G}_1,\dots,\sigma^{\G}_r\rangle$, that is the closure of the group generated by $\sigma^{\G}_1,\dots,\sigma^{\G}_r$ in $\Gal(\overline{\dsQ}/\dsQ_{(2)})$. Furthermore, $\Gal(FK/F)\simeq\langle\sigma^{\G}_1|_{K},\dots,\sigma^{\G}_r|_{K}\rangle$ for any Galois field $K$. Since by definition of $F$, $\sigma^{\G}_d,d=1,\dots,r$ acts trivially on $F$, the action of $\sigma^{\G}_d$ on $FF_{i,j}$ is determined by its action on $\dsQ_{(2)}F_{i,j}$, which was defined to be $\sigma^{\G}_d|_{\dsQ_{(2)}F_{i,j}}=\sigma_{d,1}$. We therefore get that $\Gal(FF_{i,j}/F)\simeq G_1$, and by the analogous computation, $\Gal(FE_{i,j}/F)\simeq G_2$. By \eqref{eq:infnit_count_ex1},\eqref{eq:infnit_count_ex2} the theorem is now proved.
\end{proof}
\appendix
\section{Cartan subgroups}\label{app:properties of Cartans}
In this appendix we list and prove some of the basic properties of Cartan subgroups, following the work of Mohrdieck in \cite{Mohr}. Throughout this section, the ground field is assumed to be algebraically closed. We recall the definition of a Cartan subgroup. Let $G$ be a linear algebraic group with $G^o$ reductive. A {\it{Cartan subgroup}} of $G$ is a subgroup $C<G$ satisfying the following properties
\begin{enumerate}
\item $C$ is diagonalizable.
\item $C/C^o$ is cyclic.
\item The index $[N_G(C):C]$ is finite.
\end{enumerate}
\subsection{Basic properties}
In \cite{Mohr} \S3, the following is proved
\begin{prop}\label{prop:Cartan_basic_props}
Let $G$ be as above.
\begin{enumerate}
\item
Let $h\in G$ be semisimple element. Then $h$ is contained in a Cartan subgroup. In fact, the subgroup generated by $S$ and $h$, where $S$ is a maximal torus in $Z_{G^o}(h)^o$ is a Cartan subgroup.
\item Let $C<G$ be a Cartan subgroup, and let $h\in G$ be such that $hC^o$ generates $C/C^o$. Then $C^o$ is a maximal torus of $Z_{G^o}(h)^o$.
\item\label{prop:Cartan_basic_props3} Let $C<G$ be a Cartan subgroup. Then $C^o$ is a regular torus in $G^o$, that is $Z_{G^o}(C^o)$ is a maximal torus in $G^o$.
\end{enumerate}
\end{prop}
\begin{remark}\label{rem:Cartans have regulars}
Since $Z_{G^o}(C^o)$ is a maximal torus, $T_C$, of $G^o$, we get that the the intersection of $C^o$ with the set of regular semisimple elements of $G^o$ contains an open dense subset of $C^o$. To see this, let $A(T_C)$ be the finite set of roots of $T_C$. For $\alpha\in A(T_C)$, let $u_\alpha:\mathbb G_a\to G^o$ be a homomorphism satisfying
$$
tu_\alpha(x)t^{-1}=u_\alpha(\alpha(t)x),\quad t\in T_C,x\in\mathbb G_a
$$
This shows that if $t\in\ker(\alpha)$, then there exists a unipotent element that commutes with it. Since $C^o$ is irreducible, we find that if
$$
C^o\subset\cup_{\alpha\in A(T_C)}\ker(\alpha)
$$
then there exists $\alpha\in A(T_C)$ such that $C^o\in\ker(\alpha)$, in which case $Z_{G^o}(C^o)$ will contain a unipotent, and therefore can not be a torus. Since the set $\cap_\alpha\{t\in T_C:\alpha(t)\neq1\}$ is open, contains only semisimple regular elements, and has a non empty intersection with $C^o$, we get the result.
\end{remark}
In the following we prove some properties of Cartan subgroups, generalizing the one proved in \cite{Mohr}. We first prove a structure result on non-connected algebraic groups.
\begin{prop}\label{prop:non connected cyclic structure}
Let $G$ be a linear algebraic group defined over a field $\dsF$, such that $G/G^o$ is cyclic of order $m$. Assume $\chara(\dsF)=0$ or $(\chara(\dsF),m)=1$. Then for any coset $G_i$ of $G^o$ that generates $G/G^o$, there exists $x\in G_i$ such that $G=G^o\rtimes\langle x\rangle$.
\end{prop}
\begin{proof}
Let $y\in G_i$ be any elements in $G_i$. By definition $y^m\in G^o$, and so let $t\in G^o$ be an elements such that $t^m=y^m$ (to find such $t$ write $y^m=y_uy_s$ according to the Jordan decomposition of $y^m$, and find $t_u$ in the one parameter unipotent subgroup generated by $y_u$ and $t_s$ in a maximal torus containing $y_s$ such that $t_u^m=y_u,t_s^m=y_s$. Notice that this is all possible by the assumption on $\chara(\dsF)$). Then $t$ commutes with $y$, and $yt^{-1}\in G_i$ satisfies $\left(yt^{-1}\right)^m=e$, and the result follows.
\end{proof}
\begin{prop}\label{prop:Cartan subgps pre conj props}
Let $G$ be a linear algebraic group, with $G^o$ reductive. Let $C<G$ be a Cartan subgroup, and $h\in C$ semisimple such that $hC^o$ generates $C/C^o$.
\begin{enumerate}
\item\label{prop:Cartan subgps pre conj props1} There exist a maximal torus $T_h<G^o$ containing $C^o$, and a Borel subgroup $B_h<G^o$ containing $T_h$, both invariant under conjugation by $h$.
    \item\label{prop:Cartan subgps pre conj props2} Let $T$ be a maximal torus containing $C^o$, invariant under conjugation by $h$. Then for any $t\in T$, there exists $t'\in T$ such that $\;t'tht'^{-1}\in C^oh$
    \item\label{prop:Cartan subgps pre conj props3} Let $g\in G^oh$ be a semisimple element. Then there exists $x\in G^o$ such that $x^{-1}gx\in C^oh$.
\end{enumerate}
\end{prop}
\begin{proof}
For the first part, we consider $T_h$ to be the maximal torus $Z_{G^o}(C^o)$, containing $C^o$, provided by Proposition \ref{prop:Cartan_basic_props}\eqref{prop:Cartan_basic_props3}. Let $B'$ be a Borel subgroup of $Z_{G^o}(h)$ containing $C^o$. By corollary 7.4 in \cite{St}, it is contained in a Borel subgroup $B$ of $G^o$, invariant under conjugation by $h$. Since $Z_B(C^o)$ is a maximal torus of $G^o$, it must coincide with $T_h$, and therefore we get the required result.
For the second part, notice that since $t$ and $t'$ commute, it is equivalent to showing that $[t',h]\in t^{-1}C^o$. Consider the automorphism $c_h$ of $T/C^o$ via conjugation by $h$. Then $[t',h]=t'c_h(t'^{-1})$. By definition of $C,C^o$ and $h$, we have that this automorphism has only finitely many fixed points (since any fixed point is a $C^o$ -coset inside $N_{G}(C)$, and $[N_{G^o}(C):C^o]<\infty$). Since $T/C^o$ is connected, we can apply the Lang-Steinberg Theorem (Theorem \ref{thm:Lang Steinberg} above), to get that the map $t'\mapsto t'c_h(t'^{-1})$ as a map from $T/C^o$ to itself is onto, and therefore $t^{-1}C^o$ is in the image. For the last part: by Corollary 7.5 in \cite{St} there exist a Borel subgroup $B_g$ of $G^o$ and a maximal torus $T_g<B_g$ of $G^o$ both invariant under conjugation by $g$. After conjugation by an element of $G^o$, we may have that $T_g=T_h$, and $B_g=B_h$. We therefore get that $gh^{-1}\in G^o$, preserving $T_h$ and $B_h$ and therefore $gh^{-1}\in T_h$ (since it preserves $T_h$, it is inside the normalizer of $T_h$ and since it preserves a Borel subgroup $B_h$, it is a trivial element of the (classical) Weyl group $W(G^o,T_h)$). The last part is now a consequence of the second part of the proposition.
\end{proof}
With this proposition, we can now show that Cartan subgroups are conjugated by elements of $G^o$.
\begin{cor}\label{cor:Cartan_sbgps_conjugated}
Let $G$ be as in  Proposition \ref{prop:Cartan_basic_props}. Let $h,g\in G$, be semisimple elements such that $hg^{-1}\in G^o$, and let $C_h,C_g$ be Cartan subgroups containing $h,g$ respectively, such that $C^o_hh,C^o_gg$ generate $C_h/C^o_h,C_g/C^o_g$ respectively. Then there exist an element $x\in G^o$ such that $C_h=(C_g)^x$.
\end{cor}
\begin{proof}
By Proposition \ref{prop:Cartan subgps pre conj props} \eqref{prop:Cartan subgps pre conj props3}, we get that there exists $x\in G^o$, such that $g^x\in C^o_hh$. Therefore, since $C_h$ is a Cartan subgroup containing $g^x$, and $C_h/C^o_h$ is generated by $C^o_hg^x$, by Proposition \ref{prop:Cartan_basic_props}, $C^o_h$ is a maximal torus of $Z_{G^o}(g^x)=\left(Z_{G^o}(g)\right)^x$, and therefore there exists $y\in Z_{G^o}(g)$ such that $\left(C^o_g\right)^{yx}=C^o_h$. We therefore have
$$
C_h=\langle C^o_h,h\rangle=\langle C^o_h,g^x\rangle=\langle \left(C^o_g\right)^{yx},g^{yx}\rangle=\langle C^o_g,g\rangle^{yx}=C_g^{yx}
$$
which concludes the proof.
\end{proof}
Another property that we will exploit in this paper is the structure of the Weyl group $W(G_i,C)$ of a given Cartan subgroup. For this we need some notations. Let $G$ be a semisimple linear algebraic group, $G_i$ a fixed coset of $G^o$, and $C$ a Cartan subgroup associated to $G_i$ (see \S\ref{sec:Cartan_subgp}). Let $g\in G_i\cap C$ be a fixed element, and denote by $c_g$ the automorphism of $G^o$ of conjugation with $g$. As mentioned above, $T_C:=Z_{G^o}(C^o)$ is a maximal torus of $G^o$. Denote $W:=W(G^o,T_C)$ the Weyl group of $G^o$ with respect to $T_C$, and denote by $W_C:=\left\{w\in W:w^g=w\right\}$, where $\bullet^g$ denotes the action of $g$ on $W$.
\begin{lem}\label{lem:Weyl gp strctr}
Under the above notations we have the following split exact sequence
\begin{equation}\label{eq:Weyl gp strctr}
1\to \left(T_C/C^o\right)^g\to W(G_i,C)\to W_C\to1
\end{equation}
where $\left(T_C/C^o\right)^g$ is the set of fixed points of the automorphism $c_g:T_C/C^o\to T_C/C^o$.
\end{lem}
\begin{proof}
For every $n\in N_{G^o}(C)$ we have that $n\in N_{G^o}(T_C)$, and therefore we can define the maps
$$
\phi:W(G_i,C)\to W(G^o,T_C)\quad \phi(nC^o)=nT_C
$$
Assume $nC^o\in\ker\phi$, then $nT_C=T_C$, that is $n\in T_C$. Furthermore, $n\in N_{G^o}(C)$, and so $ngn^{-1}\in gC^o$, and therefore $c_g(nC^o)=nC^o$ showing that the sequence is left exact. For right exactness, let $n\in N_{G^o}(T_C)$ such that $gng^{-1}=nt$, for some $t\in T_C$, and hence $ngn^{-1}=nt$. By Proposition \ref{prop:Cartan subgps pre conj props} \ref{prop:Cartan subgps pre conj props2} we may assume that $t\in C^o$. We therefore get that $nT_C$ preserves $C^og$. Furthermore, since $C^o=\left(Z_{T_C}(g)\right)^o$, we find that $g^{-1}n^{-1}C^ong=C^o$ and so $nT_C$ preserves $C$, and thus the map $nC^o\mapsto nT_C$ is surjective onto $W_C$. To get that the map splits, we notice that if the map $nC^o\mapsto nT_C$ satisfies the splitting. To see this, notice that since $C^o=\left(Z_{T_C}(g)\right)^o$, if $w\in W_C$ then $w$ sends $C^o\mapsto C^o$, and $C^og\mapsto C^og$, and therefore $C\mapsto C$. In particular we get that the map $nT_C\mapsto nC^o$ is well defined, injective, and we get that the exact sequence splits.
\end{proof}
\begin{cor}\label{cor:Weyl gp strctr}
Under the above notations, we have that $W(G_i,C)=\left(T/C^o\right)^g\rtimes W_C$. Furthermore, $W_C<\W(G_i,C)=N_{G^o}(C)/Z_{G^o}(C)$.
\end{cor}
This is an immediate corollary of Lemma \ref{lem:Weyl gp strctr}, and the fact that $Z_{G^o}(C)<T_C$.
\subsection{Regular semisimple elements in cosets}\label{appsec:reg ss}
For a connected linear reductive algebraic group $G^o$, an element $g\in G^o$ is called regular semisimple if its centralizer is a maximal torus, and henceforth it is contained in a unique maximal torus. It is known that the set of regular semisimple elements contains an open dense subset. In the following we show a similar result for non connected groups.
\begin{defn}\label{defn:regular semisimple}
Let $G$ be a linear algebraic group, with $G^o$ reductive. We say that a semisimple element $g\in G$ is regular, if $\left(Z_{G^o}(g)\right)^o$ is a torus.
\end{defn}
Notice that this definition generalizes the definition in the connected case, and that if $g\in G_i$, for some coset $G_i$, is regular, then by proposition \ref{prop:Cartan_basic_props} (2), $g$ is contained in a unique Cartan subgroup associated to $G_i$. We show now that most elements are regular semisimple.
\begin{prop}\label{prop:most lmnts rs}
Let $G$ be as above, $G_i$ a coset of $G^o$, and $C<G$ a Cartan subgroup associated to $G_i$ (as in \eqref{defn:Cartan_associ} in Section \ref{sec:Cartan_subgp}).
\begin{enumerate}
\item
The set of regular semisimple elements in $C\cap G_i$ contains an open dense subset.
\item
The set of regular semisimple elements in $G_i$ contains an open dense subset.
\end{enumerate}
\end{prop}
\begin{proof}
Recall that a diagonalizable linear algebraic group $D$, can be decomposed as $D\simeq D^o\times F$, where $F\simeq D/D^o$. Therefore there exists $\mu\in C\cap G_i$ such that $C\simeq C^o\times\langle\mu\rangle$. Furthermore, any element of $C\cap G_i$ is of the form $t\mu$, for $t\in C^o$. Denote by $n=\ord(\mu)$, and let
$$
U_1:=\left\{t\in C^o:\;t^n \text{ is regular}\right\}
$$
Then, $Z_{G^o}(t)=Z_{G^o}(t^n)$ for all $t\in U_1$. Notice also that by Remark \ref{rem:Cartans have regulars}, $U_1$ contains an open dense subset of $C^o$. Let $t\mu\in U_1\mu$. Then
$$
Z_{G^o}(t\mu)\subset Z_{G^o}(t^n\mu^n)=Z_{G^o}(t)
$$
and therefore if $g\in \left(Z_{G^o}(t\mu)\right)^o$, then $g\in\left(Z_{G^o}(t)\right)^o$, and therefore $g\in\left(Z_{G^o}(\mu)\right)^o$. Since $t\in U_1$ is regular semisimple, its centralizer is a maximal torus, and therefore $Z_{G^o}(t\mu)$ is diagonalizable, and its connected component is a torus, which proves part (1).
To prove (2), we use the following morphism
$$
\Phi:G^o/C^o\times C^o\to G_i\quad \Phi(gC^o,h)=gh\mu g^{-1}
$$
Notice that the dimension of both sides is the same. Moreover, since by the definition of a Cartan subgroup it has a finite index in $N_{G^o}(C)$, this map is finite to one, and so the image of $G^o/C^o\times U$, where $U\mu$ is the set of regular semisimple in $C\cap G_i$, contains an open dense subset of $G_i$.
\end{proof}
\section{$k$-Cartan subgroups and $k$-quasi irreducibility}\label{app:k-q Cartans}
Let $k$ be a number field, and let $G$ be an algebraic group defined over $k$. We say, as in Section \ref{sec:F-Cartans}, that a Cartan subgroup is a $k$-Cartan subgroup if $C$ is defined over $k$, and that $C/C^o$ is generated by $C^og$, where $g\in C(k)$. As in Section \ref{sec:diag_gps}, set $k_C$ to be the splitting field of $C$. In \cite{PR} Prasad and Rapinchuk introduced the following definitions:
\begin{defn}\label{def:k_quasi_irr_torus}
Let $G^o$ be a connected semisimple linear algebraic group defined over a number field $k$, and let $G^o=G^{(1)}\cdots G^{(r)}$ be its decomposition into an almost direct product of connected normal $k$-almost simple groups.
\begin{enumerate}
\item
We say that an element is {\it{without component of finite order}} if for some (equivalently, any) decomposition $x=x_1\cdots x_r$ with $x_i\in G^{(i)}$, all the $x_i$'s have infinite order.
\item
We say that a maximal torus $T<G$ is $k$-{\it{quasi-irreducible}} if whenever $T'<T$ is a $k$-subtorus of $T$, then $T'$ is an almost direct product of a subset of $T^{(i)}:=T\cap G^{(i)},i=1,\dots,r$.
\end{enumerate}
\end{defn}
In \cite{PR} Prasad and Rapinchuk proved that for a semisimple group, if $T$ is a $k$-quasi-irreducible anisotropic maximal torus, then any element $x\in T$ without any component of finite order, generates a dense subgroup of $T$. Furthermore, they show that if the image of $\Gal(k_T/k)$ in the group $Aut(X(T))$, as defined in \S\ref{sec:diag_gps}, contains the image of the Weyl group $W(G,T)$, then $T$ is $k$-quasi irreducible anisotropic. In the following Lemma we generalize this result for non connected groups. Let $G$ be a linear algebraic group defined over $k$, such that $G^o$ is semisimple. Fix a coset $G_i$ and denote by $G^o=G^1\cdots G^l=G^{(1,i)}\cdots G^{(r,i)}$ where the product of $G^j,j=1,\dots,l$ is the minimal decomposition of $G^o$ into $k$-simple normal subgroups and the product of $G^{(j,i)},j=1,\dots,r$ is the decomposition of $G^o$ into an almost direct product of $k$-normal subgroups invariant under conjugation by $G_i$.
\begin{defn}\label{defn:G_i-k irred}
Let $G,G^o,G_i$ and $k$ be as above. Denote $G^o=G^1\cdots G^l=G^{(1,i)}\cdots G^{(r,i)}$ be the minimal decomposition as above. We say that
\begin{enumerate}
\item
A $k$-Cartan subgroup $C$ associated to $G_i$ is $G_i-k$ \textit{quasi irreducible} if $C^o$ has no $k$-subtori other than products of a subsets of $C^o\cap G^{(j,i)},j=1,\dots,r$.
\item
An element $x\in G^o$ is {\it{without $G_i$-component of finite order}} if for some (equivalently, any) decomposition $x=x_1\cdots x_r$ with $x_i\in G^{(s,i)},s=1,\dots,r$, all the $x_i$'s have infinite order.
\end{enumerate}
\end{defn}
Following this definition we have the following generalization of the result in \cite{PR}:
\begin{lem}\label{lem:k-q_ired_Cartans}
Let $G$ be a linear algebraic group defined over a number field $k$, such that $G^o$ is semisimple. Fix a coset $G_i$ of $G^o$, and let $G^o=G^{(1,i)}\cdots G^{(r,i)}$ as above. Let $C$ be a $k$-Cartan subgroup associated to $G_i$ such that the image of $\Gal(k_C/k)$ in $Aut(X(C^o))$ contains the image of the Weyl group $W(G_i,C)$. Then $C^o$ is $G_i-k$ quasi irreducible.
\end{lem}
\begin{proof}
Let $T_C=Z_{G^o}(C^o)$ be the unique maximal torus of $G^o$ containing $C^o$, and let $R=R(G^o,T_C)$ be the root system of $G^o$ with respect to $T_C$. For a $k$-simple normal subgroup $G^j$, let $G^j=G^j_1\cdots G^j_{m_j}$ be its decomposition into absolutely almost simple subgroups, and let $R^j_s=R(G^j_s,T_C\cap G^j_s)$ be the corresponding root systems. Denote by $Y_T=X(T_C)\otimes_{\dsZ}\dsQ$, and $Y_C=X(C^o)\otimes_{\dsZ}\dsQ$. Let $g\in C\cap G_i(k)$ be an element such that $gC^o$ generates $C/C^o$. Then $C^o=\left(T_C^g\right)^o$, the connected component of the fixed points of $T_C$ under conjugation by $g$. The element $g$ defines an automorphism of $Y_T$ induced by the conjugation automorphism on $T_C$, which we denote also by $g$. Let $V^j_s$ be the subspace of $Y_T$ spanned by $R^j_s$, and $V^{(a,i)}$ the direct sum of $V^j_s$ for $G^j<G^{(a,i)},s=1,\dots,m_j$. We claim that any subspace of $Y_T$ which is invariant under $\Gal(k_C/k)$ and $g$ is a sum of some $V^{(a,i)}$. Since $C^o$ is a maximal torus in $\left(Z_{G^o}(g)\right)^o$, we find that $Y_C$ is the fixed points of $Y_T$ under the $g$-action. Furthermore, notice that $\Gal(\ck/k)$ acts on $Y_C$, and since $C^o$ splits over $k_C$, the action factors through $\Gal(k_C/k)$. Let $\Delta$ be the Dynkin diagram of $R$, and let $\tau\in Aut(\Delta)$ be an automorphism such the action of $g$ on $T$ coincides with $\tau$, and let $v\mapsto p(v)=\frac{1}{\ord(\tau)}\sum_{n=1}^{\ord(\tau)}\tau(v)$ be the projection from $Y_T$ to $Y_C$. Then $p(R)$ is a root system of $Y_C$ with Weyl group $W^\tau=\{w\in W(R):w\tau=\tau w\}$ (cf. \cite{Ca1} Proposition 13.2.2). We can now finish the proof as in \cite{PR}: Since $W^\tau$ acts irreducibly on the irreducible components of $p(R)$, we get the same for $\Gal(k_C/k)$ which we assume contains $W(H_i,C)\supset W^\tau$. Furthermore, since for a fixed $j$, $\Gal(k_C/k)$ acts transitively on $G^{(j,i)}$, we find that any $g$-invariant subspace which is invariant under $\Gal(k_C/k)$ must be the sum of some of the $V^{(a,i)}$. Now if $C'\subset C$ is a $k$-subtorus of $C^o$, then it is in particular $g$-invariant, and therefore $\ker\left(res:X(C^o)\to X(C')\right)\otimes_{\dsZ}\dsQ$ is a subspace of $Y_T$ which is invariant under $\Gal(k_C/k)$ and $g$, and thus is of the form $\oplus_{a\in A}V^{(a,i)},A\subset\{1,\dots,r\}$, and hence $C'$ is an almost direct product of $C^o\cap G^{(a,i)},a\not\in A$.
\end{proof}
The following immediate corollary of this Lemma will allow us to reduce the question of splitting fields of elements into splitting fields of $G_i-k$ quasi irreducible Cartan subgroups.
\begin{cor}\label{cor:k-q_ired_elements}
 Let $G,G_i$ and $K$ be as in Lemma \ref{lem:k-q_ired_Cartans}, and assume $C$ is a Cartan subgroup such that $W(G_i,C)\subset \Gal(k_C/K)$. Then for any regular semisimple $g\in G_i(k)\cap C(k)$ such that $g^{\ord(\tau)}$ has no $G_i$-components of finite order, the group generated by $g$ is Zariski dense in $C$ and, in particular, $\Gal(k(g)/k)=\Gal(k_C/k)$.
 \end{cor}
 \bibliographystyle{plain}
\bibliography{mybib3}
\end{document}